\documentclass[a4paper,11pt,reqno]{amsart}
\usepackage{amsmath}
\usepackage{amsthm,enumerate}

\usepackage{graphicx}
\usepackage{amssymb}

\usepackage{appendix}

\usepackage[italian,english]{babel}

\usepackage[utf8x]{inputenc}
\usepackage{fancyhdr}

\usepackage{calc}
\usepackage{url}

\usepackage[text={6in,8.6in},centering]{geometry}

\usepackage{srcltx, inputenc}

\usepackage[T1]{fontenc}

\usepackage[usenames,dvipsnames]{color}

\makeatletter
\def\cleardoublepage{\clearpage\if@twoside \ifodd\c@page\else%
         \hbox{}%
     \thispagestyle{empty}
     \newpage%
     \if@twocolumn\hbox{}\newpage\fi\fi\fi}
\makeatother

\let\cleardoublepage\clearpage

\addto\captionsitalian{}

\newtheorem{thm}{Theorem}[section]
\newtheorem{cor}[thm]{Corollary}
\newtheorem{lem}[thm]{Lemma}
\newtheorem{pro}[thm]{Proposition}
\newtheorem{defi}[thm]{Definition}
\newtheorem{rem}[thm]{Remark}
\numberwithin{equation}{section}

\newcommand{\dive}{\operatorname{div}}
\newcommand{\dx}{{\rm d}x}
\newcommand{\dy}{{\rm d}y}
\newcommand{\dz}{{\rm d}z}
\newcommand{\dt}{{\rm d}t}
\newcommand{\ds}{{\rm d}s}
\newcommand{\dv}{{\rm d}v}
\newcommand{\dS}{{\rm d}S}

\begin{document}

\title[Uniqueness of very weak solutions for a fractional filtration equation ]{Uniqueness of very weak solutions \\ for a fractional filtration equation \\ }

\author {Gabriele Grillo} \author{Matteo Muratori}
\author{Fabio Punzo}

\address {Gabriele Grillo, Matteo Muratori, Fabio Punzo: Dipartimento di Matematica, Politecnico di Milano, Piaz\-za Leonardo da Vinci 32, 20133 Milano, Italy}
\email{gabriele.grillo@polimi.it}
\email{matteo.muratori@polimi.it}
\email{fabio.punzo@polimi.it}


%
%
%
%


\maketitle

\scriptsize
\noindent{\bf Abstract}. We prove existence and uniqueness of distributional, bounded solutions to a fractional filtration equation in ${\mathbb R}^d$. With regards to uniqueness, it was shown even for more general equations in \cite{TJJ} that if two bounded solutions $u,w$ of \eqref{eq6} satisfy $u-w\in L^1({\mathbb R}^d\times(0,T))$, then $u=w$. We obtain here that this extra assumption can in fact be removed and establish uniqueness in the class of merely bounded solutions. For nonnegative initial data, we first show that a minimal solution exists and then that any other solution must coincide with it. A similar procedure is carried out for sign-changing solutions. As a consequence, distributional solutions have locally-finite energy.



\normalsize

\section{Introduction}

We consider the following problem:
\begin{equation}\label{eq6}
\begin{cases}
u_t=-(-\Delta)^s \, \Phi(u) & \text{in } \mathbb{R}^d\times \mathbb{R}^+ \, , \\
u=u_0 & \text{on } \mathbb{R}^d\times\{0\} \, ,
\end{cases}
\end{equation}
where $ s \in (0,1) $ and $ u_0 $ is a bounded initial datum. As concerns the nonlinearity $\Phi$ we assume the following condition:
\begin{equation}\label{hp-phi}\tag{H}
\Phi: \mathbb{R} \to \mathbb{R} \ \textit{is nondecreasing and locally Lipschitz} \, .
\end{equation}
We stress that \it strict \rm monotonicity of $\Phi$ is \it not \rm required. For example, the Stefan-type nonlinearity $\Phi(u)=(u-1)_+$ is admissible. Without further assumptions on $ \Phi $, such equation can be referred to as \emph{fractional filtration equation}, according to the terminology adopted in \cite{EK} in the local case $ s=1 $. A typical example is $\Phi(u)=u|u|^{m-1}$ with $m\ge1$; in this case and when $m>1$ equation \eqref{eq6} is usually known as the \it fractional porous medium equation, \rm the latter having been introduced and studied in \cite{Vaz11,Vaz12}, see also the review papers \cite{VR, VR2}.

We recall that the $s$-fractional Laplacian is defined, at least on test functions $ \varphi \in C^\infty_c(\mathbb{R}^d) $, by the formula
\begin{equation}\label{def-frac}
(-\Delta)^s \varphi(x) := c_{d,s} \ \mathrm{p.v.} \int_{\mathbb{R}^d} \frac{\varphi(x)-\varphi(x^\prime)}{|x-x^\prime|^{d+2s}} \, \dx^\prime \, ,
\end{equation}
where
$$
c_{d,s} := \frac{2^{2s}s\,\Gamma\!\left(\frac{d}{2}+s \right)}{\pi^{\frac{d}{2}}\Gamma(1-s)} \, .
$$

Well-posedness of problem \eqref{eq6} when $\Phi(u)=u|u|^{m-1}\; (m>1)$ is satisfactorily achieved in the case of \it energy solutions, \rm see \cite{Vaz12}, namely solutions belonging to a suitable fractional Sobolev space (we refer to \cite{BPSV,BSV2,PV} for uniqueness results in the linear case). However, this class in general excludes solutions corresponding to data which are merely required to be bounded. For the latter, \it distributional \rm solutions should be considered instead. In this case, well-posedness was thoroughly studied in \cite{TJJ}, and the same authors then provided successful numerical schemes for such equations in \cite{TJJ1,TJJ2}. In fact, a large class of operators and of nonlinearities was addressed in \cite{TJJ} and later in \cite{TJJ3}, where distributional, \it energy \rm solutions corresponding to data in $L^1(\mathbb{R}^d)\cap L^\infty(\mathbb{R}^d)$ are considered. As concerns the nonlinearities, it is feasible to take in \cite{TJJ} the function $\Phi(u)= \operatorname{sign}(u) |u|^m$ in the full range $m>0$, thus including the \it fractional fast diffusion equation\rm. In particular, also non-Lipschitz nonlinearities are treated.
One of the most important results of \cite{TJJ} states that two bounded solutions $u,w$ to \eqref{eq6} coincide up to some time $ T>0 $ under the extra assumption
\begin{equation}\label{L1}
u-w\in L^1({\mathbb R}^d\times(0,T)) \, .
\end{equation}
This kind of condition first appeared in the \it local \rm case $s=1$ in \cite{BC}. On the other hand, it should be noted that \eqref{L1} was later proved to be inessential, see \cite{BCP}.

Our main goal here is to show, by using a strategy of proof which is completely different from the one of \cite{TJJ} since it involves the extension method of Caffarelli-Silvestre \cite{CS}, that condition \eqref{L1} can be dropped also in the \emph{nonlocal} case, upon assuming that $\Phi$ sa\-tisfies hypothesis \eqref{hp-phi}. This yields uniqueness in the natural class of merely bounded solutions, see Theorem \ref{teouni}. Such a result is achieved by first proving a delicate comparison principle in Euclidean balls, which is then used to establish existence of a \it minimal \rm nonnegative solution for nonnegative initial data. Subsequently, we prove that all solution to \eqref{eq6} corresponding to a given bounded, nonnegative initial datum coincides with the minimal one. Since the minimal solution is shown to have \emph{locally-finite energy} (see Corollary \ref{ext}), this clearly entails that there exist no purely distributional solutions, i.e.~any distributional solution has a locally-finite energy. Sign-changing solutions are addressed afterwards, by similar techniques that allow us to resort to the case of nonnegative solutions.

In the pure porous-medium case, namely for $\Phi(u)=u|u|^{m-1}$ with $ m>1 $, existence of energy solutions was shown in \cite{Vaz12} for $L^1$ data (see \cite{BFV,BSV, BV1,BV2} for the same equation studied on regular domains), whereas existence of distributional (energy) solutions for more general $\Phi$ and operators more general than $ (-\Delta)^s $ was proved in \cite{TJJ, TJJ3}, at least for data in $L^1(\mathbb{R}^d)  \cap L^\infty(\mathbb{R}^d) $. After the first version of the present paper was completed, in \cite{AiB} the authors proved that uniqueness holds for problem \eqref{eq6} without requiring that $ \Phi $ is locally Lipschitz, but in the significantly smaller class of \emph{entropy solutions}.

We remark that fractional, nonlinear diffusion problems of the type studied here arise in several applied models, for example and without any claim of completeness we mention that crossovers between fractional and local diffusions are investigated e.g.\ in \cite{BGT, LMT}, that hydrodynamic limits of particle systems with long-range dynamics lead to fractional diffusion equations which can be either linear or nonlinear, see e.g.\ \cite{JKO,JLS}, and that such kind of equations also arise in boundary heat control \cite{AC}.

\subsection{Existence and uniqueness results}

We start by introducing the definition of distributional, or very weak, solution to problem \eqref{eq6}.

\begin{defi}\label{defsol2}
Let $ s \in (0,1) $, $ u_0 \in L^\infty(\mathbb{R}^d) $ and $ \Phi $ satisfy \eqref{hp-phi}. We say that a measurable function $u$ is a \emph{very weak} solution to problem \eqref{eq6}  if $u\in L^\infty(\mathbb{R}^d\times \mathbb{R}^+)$ and for a.e.~$T>0$ there holds
\begin{equation}\label{eq7}
\begin{aligned}
\int_0^T\int_{\mathbb{R}^d} u \, \varphi_t
\,\dx\dt = & \int_0^T\int_{\mathbb{R}^d}\Phi(u) \, (-\Delta)^s\varphi\,\dx\dt \\
& +\int_{\mathbb{R}^d}u(x,T) \, \varphi(x,T)\,\dx-\int_{\mathbb{R}^d}u_0(x) \, \varphi(x,0)\,\dx
\end{aligned}
\end{equation}
for every $\varphi\in C_c^{\infty}\!\left(\mathbb{R}^d \times[0,T]\right)$.
\end{defi}

From here on by ``solution'' to \eqref{eq6} we will implicitly mean a very weak solution in the sense of Definition \ref{defsol2}, unless otherwise specified.

\smallskip
Our main existence and uniqueness results are the following.

\begin{thm}[Existence of bounded solutions]\label{thmesi-sign-ch}
Let $ s \in (0,1) $, $ u_0 \in L^\infty(\mathbb{R}^d) $ and $ \Phi $ satisfy \eqref{hp-phi}. Then problem \eqref{eq6} admits a solution according to Definition \ref{defsol2}.
\end{thm}

\begin{thm}[Uniqueness of bounded solutions]\label{teouni}
Let the assumptions of Theorem \ref{thmesi-sign-ch} hold. Then the solution to \eqref{eq6} is unique.
\end{thm}

It is well known that fractional Laplacians can be represented through suitable \emph{extension} operators, see \cite{CS,ST}. Indeed, let $v\in L^\infty(\mathbb{R}^d)$ and consider its $2s$-harmonic extension $E(v)$ given by
\[
E(v)(\cdot,y)=P_s(\cdot,y)\ast v \qquad \forall y>0 \, ,
\]
where
\[
P_s(x,y) := \kappa_{d,s} \, \frac{y^{2s}}{\left|(x,y)\right|^{d+2s}} \qquad \forall x \in {\mathbb R}^d  \, , \ \forall y>0 \, , \qquad \kappa_{d,s} := \left[ \int_{\mathbb{R}^d} \left|(x,1)\right|^{-d-2s} \dx \right]^{-1}
\]
is the Poisson kernel of $ (-\Delta)^s $, see \cite{CSi,CS}. It will be shown that, if $ u $ is the solution constructed in Theorem \ref{thmesi-sign-ch}, then there holds
\begin{equation}\label{local energy}
\nabla E(\Phi(u))\in L^2_{\rm{loc}}\!\left(\overline{\Omega} \times [0,\infty)\right) ,
\end{equation}
where $\Omega:={\mathbb R}^d\times(0,\infty)$, namely $u$ turns out to be what is usually referred to as a \it local weak \rm solution (see Remark \ref{oss2}). Hence, in view of Theorem \ref{teouni}, it follows that \emph{any} very weak solution is in fact a local weak solution. More precisely, we have the following.

\smallskip
In the sequel, we will denote the fractional Sobolev space $ \dot{H}^s(\mathbb{R}^d) $ as the closure of $ C_c^\infty(\mathbb{R}^d) $ with respect to the norm $ \| (-\Delta)^{{s}/{2}}(\cdot) \|_{L^2(\mathbb{R}^d)} $. In particular, it is easy to see that, as a consequence of \eqref{local energy}, for any cut-off function $ \gamma_R $ as in \eqref{eq:cutoff} there holds
\begin{equation}\label{hsloc}
\gamma_R \, \Phi(u) \in L^2_{\mathrm{loc}}([0,\infty);\dot{H}^s(\mathbb{R}^d)) \, .
\end{equation}

\begin{cor}\label{ext}
Let the assumptions of Theorem \ref{thmesi-sign-ch} hold and let $u$ be a very weak solution to problem \eqref{eq6} according to Definition \ref{defsol2}. Then $u$ is a local weak solution, in the sense that property \eqref{local energy} holds. Hence, also \eqref{hsloc} is satisfied.
\end{cor}


We comment that a crucial step in our method of proof consists of showing the existence of a \it minimal \rm solution to problem \eqref{eq6}, for nonnegative data and for $\Phi$ satisfying also $\Phi(0)=0$.
\begin{pro}[Existence of the minimal solution for nonnegative data]\label{thmesi}
Let $ s \in (0,1) $, $ u_0 \in L^\infty(\mathbb{R}^d) $ with $ u_0 \ge 0 $ and let $ \Phi $ satisfy \eqref{hp-phi} with $\Phi(0)=0$. Then problem \eqref{eq6} admits a \emph{minimal} solution, i.e.~there exists a {nonnegative} solution $ \underline{u} $ to \eqref{eq6} such that, if $ u $ is any {nonnegative} solution to \eqref{eq6} according to Definition \ref{defsol2}, there holds $ \underline{u} \le u $.
\end{pro}

We point out that, a posteriori, as a consequence of Theorem \ref{teouni} any nonnegative solution in fact coincides with the minimal one. However, we preferred to state its existence in a separate proposition since it is a key tool to our strategy.


\subsection{The strategy of proof}

Our uniqueness proof, which is carried out in Section \ref{proofs}, is mainly inspired from an argument that crucially relies on the existence of the minimal solution in the case of nonnegative initial data. A related approach was adopted in the local case in \cite{KKT,Pu}. However, since the $s$-fractional Laplacian of nontrivial compactly-supported functions is not compactly supported, in our estimates some further terms to be controlled appear when dealing with cut-off arguments. In order to manage them, we need some refined estimates on the behavior of the fractional Laplacian (and of a related nonlinear nonlocal operator) of cut-off functions, see the first part of the proof of Theorem \ref{teouni} in Subsection \ref{nonneg}. We stress that we will first deal with \it nonnegative \rm solutions in Subsection \ref{nonneg} and then extend the uniqueness result to general bounded solutions in Subsection \ref{sign-change}.

The construction of the minimal solution is based on a comparison principle for distributional solutions in balls, which in the local case is proved by means of the so-called duality method, first introduced in \cite{ACP, Pierre} and then exploited in several frameworks, both local and nonlocal, see e.g.~\cite{GMP,GMPjmpa,GMPjems}. In our setting, comparison occurs between solutions to problems in ${\mathbb R}^d$ and solutions to problems involving the \it spectral fractional Laplacian \rm in Euclidean balls. This is in fact the most delicate point of our paper, to which we devote the entire Section \ref{mini} and part of Section \ref{sec:pre}, where some related preliminary results are discussed (see also Appendix \ref{app}). Once existence of a (minimal) solution for nonnegative initial data is guaranteed, still in Subsection \ref{sign-change} we briefly show  how to use such a result to prove existence for general bounded data.

Since the (spectral) fractional Laplacian operator on different domains acts differently on fixed test functions, in order to prove the comparison principle we first have to show a key integral inequality, which is the content of Proposition \ref{ineq} and strongly relies on the extension operator. Furthermore, we need to consider solutions to suitable backward fractional parabolic problems in balls (see Proposition \ref{dual}), that to our knowledge have not been much studied in the literature. In particular, a key trace property of the conormal derivative of the extension of such solutions is established, the latter having a fundamental role (see Lemma \ref{prop-compact}).


\section{Preliminaries: extension problems} \label{sec:pre}

As mentioned in the Introduction, in view of the seminal paper \cite{CS} we know that the fractional Laplacian operator is strictly connected with a suitable \emph{extension problem} defined on the upper half space, the latter being denoted by
\begin{equation*}
\Omega:=\left\{(x,y) \in \mathbb R^{d+1}: \ x \in \mathbb{R}^d \, , \  y>0\right\} .
\end{equation*}
In addition, for all $R>0$ we define the Euclidean balls
\[
B_R := \left\{ x \in \mathbb R^d : \ |x|<R \right\}
\]
along with the upper ``cylinders''
\begin{equation*}
\mathcal{C}_R:= \left\{ (x,y) \in \Omega : \ |x|<R \right\}
\end{equation*}
and the upper ``half balls''
\begin{equation}\label{eq10}
\Omega_R := \left\{ (x,y) \in \Omega : \ |(x,y)|<R \right\}.
\end{equation}

In the sequel we will mostly, but not only, deal with extensions to the half space of bounded functions.

\begin{defi}
Let  $ s \in (0,1) $ and $ v \in L^\infty(\mathbb{R}^d) $. Then its \emph{$2s$-harmonic extension} $E(v): \Omega\to\mathbb{R}$ is defined as
\begin{equation}\label{eq8}
E(v)(\cdot,y)=P_s(\cdot,y)\ast v \qquad \forall y>0 \, ,
\end{equation}
where the generalized \emph{Poisson kernel} $P_s$ is given by
\begin{equation}\label{eq9}
P_s(x,y) := \kappa_{d,s} \, \frac{y^{2s}}{\left|(x,y)\right|^{d+2s}} \qquad \forall (x,y) \in \Omega \, .
\end{equation}
\end{defi}
It is not difficult to check that $P_s$ is $2s$-harmonic, in the sense that
$$
\dive\!\left( y^{1-2s}\nabla P_s \right) = 0 \, ,
$$
and satisfies
$$
\lim_{y \downarrow 0} P_s(\cdot,y) = \delta_0 \qquad \text{in } \mathcal{D}^\prime(\mathbb{R}^d) \, ,
$$
which at least formally shows the identity $ E(v)(x,0)=v(x) $, so that $ E(v) $ is indeed an extension of $v$ to the half space. In the next lemma we collect some useful properties of the extension of bounded functions, which will be useful below and can easily be deduced from \eqref{eq8} (we omit the simple proof).
\begin{lem}\label{ext-prop}
Let  $ s \in (0,1) $ and $ v \in L^\infty(\mathbb{R}^d) $. Then $E(v)$ is a smooth function in $ \Omega $ such that
\begin{equation*}
\dive\!\left( y^{1-2s}\nabla E(v) \right) = 0 \qquad \text{in } \Omega
\end{equation*}
and
\begin{equation*}
\lim_{y \downarrow 0} E(v)(\cdot,y) = v \qquad \text{in } L^p_{\mathrm{loc}}(\mathbb{R}^d) \, , \ \text{for every } p \in [1,\infty) \, .
\end{equation*}
\end{lem}
In particular $E(v)$ solves the problem
\begin{equation}\label{eq13}
\begin{cases}
\dive\!\left( y^{1-2s}\nabla E(v) \right) = 0  & \text{in } \Omega \, , \\
E(v)(\cdot,0) = v & \text{in } \mathbb{R}^d \, ,
\end{cases}
\end{equation}
whence the meaning of ``$2s$-harmonic'' extension, the case $ s=1/2 $ corresponding to the usual notion of harmonicity. It is by now well known, see again \cite{CS}, that upon setting
\begin{equation}\label{eq1}
\frac{\partial E (v)}{\partial y^{2s}}:=\mu_s \, \lim_{y\downarrow 0} y^{1-2s}\frac{\partial E(v)(\cdot,y)}{\partial y}
\end{equation}
with
\begin{equation}\label{eq0}
\mu_s:=\frac{2^{2s-1}\Gamma(s)}{\Gamma(1-s)} ,
\end{equation}
the key identity
\begin{equation}\label{eq2}
-\frac{\partial E(v)}{\partial y^{2s}}=(-\Delta)^s v
\end{equation}
holds, namely the $s$-fractional Laplacian coincides with the so called ``Dirichlet-to-Neumann'' operator of the extension problem \eqref{eq13}. Clearly, in the present framework, definition \eqref{eq1} and identity \eqref{eq2} are purely formal; nevertheless it can be shown that they hold at least in the distributional sense. In fact here we will only use them applied to smooth test functions.

\begin{lem}\label{prop-compact}
Let $s \in (0,1)$ and $ \varphi \in C^\infty_c(B_R) $ for some $R>0$. Then
\begin{equation}\label{u1}
\lim_{y\downarrow 0} E(\varphi)(x,y) = \varphi(x) \qquad \forall x \in \mathbb{R}^d
\end{equation}
and
\begin{equation}\label{u2}
-\mu_s \, \lim_{y\downarrow 0} y^{1-2s}\frac{\partial E(\varphi)(x,y)}{\partial y} = (-\Delta)^s \varphi (x) \qquad \forall x \in \mathbb{R}^d \, ,
\end{equation}
both limits being uniform in $x$. Moreover, the following estimates hold:
\begin{equation}\label{eq52}
\left| E(\varphi)(x,y) \right| \leq C \frac{y^{2s}}{|(x,y)|^{d+2s}}\qquad  \forall (x,y) \in \Omega \setminus \Omega_{R} \, ,
\end{equation}
\begin{equation}\label{eq53}
\left| \nabla E(\varphi)(x,y) \right| \leq C\frac{y^{2s-1}}{|(x,y)|^{d+2s}} \qquad \forall (x,y) \in \Omega \setminus \Omega_{R} \, ,
\end{equation}
for a suitable constant $C>0$ depending only on $ d,s,\varphi,R $.
\end{lem}
As concerns \eqref{u1}--\eqref{u2}, we refer to \cite[Section 3.1]{CS}. On the other hand, formulas \eqref{eq52} and \eqref{eq53} follow easily from \eqref{eq8} and \eqref{eq9}.

There are several possible definitions of $s$-fractional Laplacian in a Euclidean ball (or more generally in a bounded regular domain): see \cite{BFV} and references therein. To our purposes it is more convenient to use the \emph{spectral} one, which is again crucially related to a suitable localised extension operator that we will describe below. Indeed, given $R>0$ and any $ \varphi \in C^\infty_c(B_R) $, by definition the spectral $s$-fractional Laplacian of $ \varphi $ is the $s$-th power of the operator $ (-\Delta)$, that is
\begin{equation}\label{eq305}
(-\Delta)_R^s \, \varphi(x) := \sum_{k=1}^\infty \lambda_k^s \, \hat \varphi_k \, \phi_k(x) \qquad \forall x \in B_R \, ,
\end{equation}
where $ \{ \phi_k \}_{k \in \mathbb{N}} $ is a sequence of eigenfunctions of $(-\Delta)$ completed with homogeneous Dirichlet boundary conditions on $ \partial B_R $ associated to a corresponding nondecreasing sequence of positive eigenvalues $ \{ \lambda_k \}_{k \in \mathbb{N}} $, which form an orthonormal basis for $ L^2(B_R) $, and $ \{ \hat \varphi_k \}_{k \in \mathbb{N}} $ is the respective sequence of Fourier coefficients of $ \varphi $. By virtue of \eqref{eq305} it comes natural to introduce the space $ H^s_0(B_R) $, namely the one formed by functions $ \varphi \in L^2(B_R) $ such that
\begin{equation}\label{Hs}
\left\| \varphi \right\|_{H^s_0(B_R)}^2:= \left\| (-\Delta)^{s/2}_R \varphi \right\|_{L^2(B_R)}^2  = \sum_{k=1}^\infty \lambda_k^s \, \hat \varphi_k^2 < \infty \, ,
\end{equation}
i.e.~those functions in $ L^2(B_R) $ whose (distributional) $s/2$-fractional Laplacian is also in $L^2(B_R)$. Similarly, we define the \emph{domain} of the $s$-fractional Laplacian in $B_R$, which we denote by $ \mathrm{Dom}(-\Delta)^s_R $, as the space of functions in $ L^2(B_R) $ such that $ (-\Delta)^s_R \, \varphi \in L^2 (B_R) $, which amounts to requiring
\begin{equation}\label{H2s}
\left\| \varphi \right\|_{\mathrm{Dom}(-\Delta)_R^s}^2 := \left\| (-\Delta)^{s}_R \, \varphi \right\|_{L^2(B_R)}^2 = \sum_{k=1}^\infty \lambda_k^{2s} \, \hat \varphi_k^2 < \infty \, .
\end{equation}
In particular, it is straightforward to check that if $ \varphi_1 \in H^s_0(B_R) $ and $ \varphi_2 \in \mathrm{Dom}(-\Delta)^s_R  $ then
\begin{equation}\label{eq401}
\int_{B_R}(-\Delta)^{s/2}_R \varphi_1 \, (-\Delta)^{s/2}_R \varphi_2 \, \dx =  \int_{B_R} \varphi_1 \, (-\Delta)^{s}_R \, \varphi_2 \, \dx \, .
\end{equation}
The ``spectral'' extension $E_R(\varphi)$ of $ \varphi $ is formally defined as the solution to the following problem:
\begin{equation}\label{eq14}
\begin{cases}
\dive\!\left( y^{1-2s} \nabla E_R(\varphi) \right) =0 &\text{in } \mathcal C_R \, , \\
E_R(\varphi) = 0 & \text{on } \partial\mathcal C_R\cap \{y>0\} \, , \\
E_R(\varphi) = \varphi & \text{on } B_R \times \{ 0 \}  \, .
\end{cases}
\end{equation}
It can be proved that there exists a unique solution to \eqref{eq14}, at least for any $ \varphi \in H^s_0(B_R) $, belonging to the space $ X^s_0(\mathcal{C}_R) $, the latter being the closure of $ C^\infty_c(\mathcal{C}_R \cup \left( B_R \times \{ y=0 \} \right) ) $ with respect to the norm
$$
\left\| \nabla f \right\|_{L^2\left(\mathcal{C}_R;y^{1-2s}\dx\dy\right)}^2 = \int_{\mathcal{C}_R} \left| \nabla f(x,y) \right|^2 y^{1-2s} \dx\dy \qquad \forall f \in C^\infty_c(\mathcal{C}_R \cup \left( B_R \times \{ y=0 \} \right) ) \, .
$$
By elliptic regularity $ E_R (\varphi) $ is a smooth function on $ \overline{\mathcal{C}}_R \cap \{ y>0 \} $ (for any $ \varphi \in L^2(B_R) $ actually). See e.g.~\cite{CSi}, and references therein.

In the sequel, for our strategy to work it is crucial to be able to apply the analogues of \eqref{u1}--\eqref{u2} (in balls) when $ \varphi $ and $ (-\Delta)^s_R \, \varphi $ merely belong to $ L^2(B_R) $. They are ensured by the following technical lemma, whose proof is deferred to Appendix \ref{Appnochetto}.

\begin{lem}\label{lemma-nochetto}
Let $ s \in (0,1) $, $R>0$, $ \varphi \in L^2(B_R) $ and $ E_R(\varphi) $ be the solution to \eqref{eq14}, defined by \eqref{eq300}. Then
\begin{equation}\label{limit-0}
\lim_{y\downarrow0} E_R(\varphi)(\cdot,y) = \varphi \qquad \text{in } L^2(B_R)
\end{equation}
and the estimates
\begin{equation}\label{stima-L2}
\left\| E_R(\varphi) (\cdot,y) \right\|_{L^2(B_R)} \le C_\theta \, e^{-\theta\sqrt{\lambda_1}y}  \left\| \varphi \right\|_{L^2(B_R)} \qquad \forall y > 0 \, , \quad \forall \theta \in (0,1) \,,
\end{equation}
\begin{equation}\label{stima-Hs}
\left\| \nabla E_R(\varphi) (\cdot,y) \right\|_{L^2(B_R)} \le C_\theta \, e^{-\theta\sqrt{\lambda_1}y} \left\| \varphi \right\|_{L^2(B_R)} \qquad \forall y \ge 1 \, , \quad \forall \theta \in (0,1) \, ,
\end{equation}
hold, where $ C_\theta>0 $ depends only on $ \theta,\lambda_1, s$. If in addition $ \varphi \in H^s_0(B_R) $, then
\begin{equation}\label{stima-integrale}
\sqrt{\mu_s} \left\| \nabla E_R(\varphi) \right\|_{L^2(\mathcal{C}_R;y^{1-2s}\dx\dy)} = \left\| \varphi \right\|_{H^s_0(B_R)} .
\end{equation}
Furthermore, if $ \varphi \in \mathrm{Dom}(-\Delta)^s_R $ there holds
\begin{equation}\label{limit}
- \mu_s\lim_{y\downarrow0} y^{1-2s}\frac{\partial E_R(\varphi)(\cdot,y)}{\partial y} = (-\Delta)_R^s\,\varphi \qquad \text{in } L^2(B_R) \, .
\end{equation}
\end{lem}

\section{Comparison principles in balls}\label{mini}

In order to construct a minimal solution to \eqref{eq6} when $ u_0 \ge 0 $, namely to prove Proposition \ref{thmesi}, first of all it is essential to be able to suitably approximate the latter by analogous problems posed in balls, with \emph{homogeneous} Dirichlet boundary conditions. In the sequel, we will make the additional assumption $\Phi(0)=0$, since it is important for the following strategy that $ u \ge 0 $ implies $ \Phi(u) \ge 0 $. We will then explain in the proofs of our main results how this extra requirement can be dropped. 

\begin{defi}\label{D1f}
Let $ s \in (0,1) $, $R \ge 1 $, $ u_0 \in L^\infty(B_R) $ with $ u_0 \ge 0 $ and $ \Phi $ satisfy \eqref{hp-phi}  with $\Phi(0)=0$. We say that a measurable function $u_R$ is a solution to problem
\begin{equation}\label{eq17f}
\begin{cases}
(u_R)_t = -(-\Delta)^s_R \, \Phi(u_R) & \text{in } B_R \times \mathbb{R}^+ \, , \\
\Phi(u_R)= 0  & \text{on } \partial B_R\times \mathbb{R}^+ \, , \\
u_R=u_0 & \text{on }  B_R \times \{0\} \, ,
\end{cases}
\end{equation}
if $ u_R \in L^\infty(B_R \times \mathbb{R}^+)$, $u_R\ge0$ and for a.e.~$T>0$ there holds
\begin{equation}\label{eq18f}
\begin{aligned}
\int_0^T \int_{B_R} u_R \, \psi_t \, \dx \dt = & \, \int_0^T\int_{B_R} \Phi(u_R) \, (-\Delta)^s_R\,\psi \, \dx \dt \\
& \, + \int_{B_R} u_R(x,T) \, \psi(x,T)\,\dx - \int_{B_R} u_0(x) \, \psi(x,0)\,\dx
\end{aligned}
\end{equation}
for every $\psi \in C^1\!\left([0, T]; L^2(B_R)\right)\cap L^2\!\left((0, T);\operatorname{Dom}(-\Delta)_R^s\right)$.
\end{defi}

%

We comment that problem \eqref{eq17f}, on bounded domains, was studied previously in \cite[Definition 3.1]{BFV}, by using a different concept of solution.

\smallskip

The main ``local comparison'' result we aim at establishing in this section, which will be crucial to the proof of Proposition \ref{thmesi}, is the following.
\begin{pro}\label{comppr}
Let $ s \in (0,1) $, $R \ge 1 $, $ u_0 \in L^\infty(B_R) $ with $ u_0 \ge 0 $ and $ \Phi $ satisfy \eqref{hp-phi}  with $\Phi(0)=0$. Let $u$ be a \emph{nonnegative} very weak solution to \eqref{eq6} in the sense of Definition \ref{defsol2} and $u_R$ be a solution to \eqref{eq17f} in the sense of Definition \ref{D1f}. Then $ u \ge u_R $ a.e.~in $ B_R \times \mathbb{R}^+ $.
\end{pro}

In order to prove Proposition \ref{comppr} we first need to show some technical facts concerning the ``extended'' versions of problems \eqref{eq6} and \eqref{eq17f}. This will be the content of the next subsection. Then the proof of Proposition \ref{comppr} itself will be carried out in Subsection \ref{comppr}.

\subsection{Auxiliary results}

We start by establishing an important consequence of \emph{local energy estimates}, the latter being true for a suitable class of functions.

\begin{lem}\label{lem6}
Let $ R \ge 1 $ and $v\in L^{\infty}(\mathbb R^d)$ with
\begin{equation*}
\nabla E(v)\in L^2\!\left(\mathcal C_R ; e^{-\sqrt{\lambda_1}y} y^{1-2s} \dx\dy \right) \cap  L^2_{\mathrm{loc}}\!\left(  \overline{\Omega} ; y^{1-2s} \dx\dy \right) ,
\end{equation*}
where $ \lambda_1>0 $ is the first eigenvalue of $ (-\Delta)_R^s $. Then the identity
\begin{equation}\label{eq250}
\int_{\mathbb R^d} v \, (-\Delta)^s\varphi\, \dx =\mu_s \int_{\mathcal C_R} \langle \nabla E(v), \nabla E_R(\varphi)\rangle\, y^{1-2s} \dx \dy
\end{equation}
holds for every $\varphi\in C^{\infty}_c(B_R)$.
\end{lem}

\begin{lem}\label{lem7}
Let $v\in L^{\infty}(\mathbb R^d)$ with $(-\Delta)^s v\in L^\infty(\mathbb R^d)$, in the distributional sense. Let $\alpha>0$ and $ r \ge 1 $. Then the following energy estimate holds:
\begin{equation}\label{est-energy}
\begin{aligned}
\int_{\mathcal{C}_r} \left| \nabla E(v) \right|^2 \rho_\alpha(y) \, y^{1-2s}  \, \dx \dy
\le  \frac{2}{\mu_s} \int_{B_{2r}}  v \, (-\Delta)^s v \, \dx +  4C^2 \, \| v \|_\infty^2 \left|B_{2r}\right| \int_0^{\infty} \rho_\alpha(y) \, y^{1-2s}  \, \dy \, ,
\end{aligned}
\end{equation}
where $ \rho_\alpha(y) $ is any regular positive function on $ [0,\infty) $ satisfying
\begin{equation}\label{eq-rho-1}
\left|\rho_\alpha^\prime(y) \right| \le C \, \rho_\alpha(y) \qquad \forall y \ge 0
\end{equation}
and
\begin{equation}\label{eq-rho-2}
\rho_\alpha(y) = 1 \quad \forall y \in [0,1] \, , \qquad
c \, e^{-\alpha y} \le \rho_\alpha(y) \le C \, e^{-\alpha y} \quad \forall y \ge 1 \, ,
\end{equation}
for suitable positive constants $c,C $, with $C$ large enough.
\end{lem}

The proofs of Lemmas \ref{lem6} and \ref{lem7} will be given in Appendix  \ref{App2} and \ref{App3}, respectively. The aim of the next result is to show that a nonnegative solution to \eqref{eq6} is in fact a supersolution to problem \eqref{eq17f}.

\begin{pro}\label{ineq}
Let $ \Phi $ satisfy \eqref{hp-phi} with $\Phi(0)=0$. Let $u$ be a nonnegative very weak solution to problem \eqref{eq6}, in the sense of Definition \ref{defsol2}. Then, for a.e.~$T> 0$ and every $ R\ge1 $, there holds
\begin{equation}\label{eq18f2}
\begin{aligned}
\int_0^T \int_{B_R} u \, \psi_t \,\dx \dt \le & \,  \int_0^T\int_{B_R} \Phi(u) \, (-\Delta)^s_R\,\psi \, \dx \dt \\
& \, + \int_{B_R} u(x,T) \, \psi(x,T)\,\dx - \int_{B_R} u_0(x) \, \psi(x,0)\,\dx
\end{aligned}
\end{equation}
for any nonnegative $\psi \in C^1\!\left([0, T]; L^2(B_R)\right)\cap L^2\!\left((0, T);\operatorname{Dom}(-\Delta)_R^s\right)$.
\end{pro}
\begin{proof}
To begin with, without loss of generality, we assume that
\begin{equation}\label{ass-1}
\nabla E(\Phi(u))\in L^2\!\left(\mathcal C_R \times [0,T] ; e^{-\sqrt{\lambda_1}y} y^{1-2s} \dx\dy\dt \right) \cap  L^2_{\mathrm{loc}}\!\left(  \overline{\Omega} \times [0,T] ; y^{1-2s} \dx\dy\dt \right) .
\end{equation}
We will explain at the end of the proof how it is possible to get rid of such hypothesis. Hence, by virtue of Lemma \ref{lem6} with $v\equiv \Phi(u)(\cdot, t)$, for a.e.~$t>0$ we obtain the identity
\begin{equation}\label{eq403}
\int_{\mathbb R^d} \Phi(u)(x, t) \,(-\Delta)^s\varphi(x,t) \, \dx  = \mu_s \int_{\mathcal C_R} \left \langle \nabla E(\Phi(u))(x, y, t) , \nabla E_R(\varphi)(x,y,t) \right \rangle y^{1-2s} \dx \dy\, ,
\end{equation}
valid for every $ \varphi \in C^\infty_c(B_R\times [0, T])$. A time integration of \eqref{eq403} that takes advantage of \eqref{eq7} yields
\begin{equation}\label{eq230f}
\begin{aligned}
\int_0^T\int_{B_R} u \, \varphi_t \,\dx\dt = & \, \mu_s \int_0^T  \int_{\mathcal C_R} \left \langle \nabla E(\Phi(u)) , \nabla E_R(\varphi) \right \rangle y^{1-2s} \dx \dy \dt \\
& \, + \int_{B_R} u(x,T) \, \varphi(x,T)\,\dx - \int_{B_R} u_0(x) \, \varphi(x,0)\,\dx \, .
\end{aligned}
\end{equation}
Given any $ \psi $ as in \eqref{eq18f2}, one can find a sequence $ \{ \varphi_n \}_{n \in \mathbb{N}} \subset C^\infty_c(B_R\times [0, T]) $ such that
$$
\varphi_n \xrightarrow[n\to\infty]{} \psi \quad \text{in } L^2\!\left([0,T]; H^s_0(B_R) \right) \qquad \text{and} \qquad \left(\varphi_n\right)_t \xrightarrow[n\to\infty]{} \psi_t \quad \text{in } L^2\!\left([0,T]; L^2(B_R) \right) .
$$
This can be achieved, for instance, by combining a time convolution of $ \psi $ and the density of $ C_c^\infty(B_R) $ in $ H^s_0(B_R) $ (we omit details). As a consequence, upon applying \eqref{eq230f} to $ \varphi \equiv \varphi_n $, letting $ n \to \infty $ and exploiting \eqref{stima-Hs}, \eqref{stima-integrale}, \eqref{ass-1}, we can deduce the validity of the same identity with $ \varphi = \psi $.

Now we focus on the integral term in \eqref{eq230f} involving gradients. Let $ \xi$ be defined as in \eqref{cutoff-y} and, correspondingly, for all $ k \in \mathbb{N} $ put
\begin{equation}\label{cutoff-y-k}
\xi_k(y) := \xi\!\left( \frac{y}{k} \right) \qquad \forall y \ge 0 \, .
\end{equation}
Furthermore, for every $ \varepsilon>0 $ let us introduce the ``lifted'' cylinder
$$
\mathcal C_{R, \varepsilon}:= \mathcal C_R \cap \{y>\varepsilon\} \, .
$$
Note that the extended functions $ E(\Phi(u)) $ and $ E_R(\psi) $ are smooth in $ \mathcal C_{R, \varepsilon} $ but they need not be in $\mathcal C_{R}$. In the next passages we will omit explicit time dependence, in order to lighten notations. For a.e.~$t\in (0, T)$ and every $ k>\varepsilon $, a standard integration by parts reveals that
\begin{equation}\label{eq235f}
\begin{aligned}
& \int_{\mathcal C_{R, \varepsilon}} \left \langle \nabla E(\Phi(u)) , \nabla \! \left(E_R(\psi) \, \xi_k \right) \right \rangle y^{1-2s} \dx \dy \\
 = & - \int_{\mathcal C_{R, \varepsilon}} E(\Phi(u)) \dive\!\left[ y^{1-2s} \, \nabla \!\left(E_R(\psi)\,\xi_k\right) \right] \dx \dy  - \int_{B_R\times\{\varepsilon\}} E(\Phi(u)) \, y^{1-2s} \, \frac{\partial E_R(\psi)}{\partial y}\, \dx \\
& + \int_{\partial \mathcal C_R\cap \{y>\varepsilon\}} E(\Phi(u)) \, \xi_k \, y^{1-2s} \, \frac{\partial E_R(\psi)}{\partial {\bf n}}\, \dS \,
\\
\le & - \int_{\mathcal C_{R, \varepsilon}} E(\Phi(u)) \dive\!\left[ y^{1-2s} \, \nabla \!\left(E_R(\psi)\xi_k\right) \right] \dx \dy - \int_{B_R\times\{\varepsilon\}} E(\Phi(u)) \, y^{1-2s} \, \frac{\partial E_R(\psi)}{\partial y}\, \dx \, , \\
\end{aligned}
\end{equation}
where $ \mathbf{n} $ and $ \dS $ stand for the outer unit normal and the $d$-dimensional Hausdorff measure on $ \partial \mathcal{C}_R $, respectively. Note that in \eqref{eq235f} we have used the property
$$\frac{\partial E_R(\psi)}{\partial {\bf n}} \leq 0 \qquad \text{on } \partial \mathcal{C}_R \, , $$
which holds because $ E_R(\psi) $ is nonnegative in $ \mathcal{C}_R $ and vanishes on $ \partial \mathcal{C}_R $. On the other hand, since
\[
\dive\!\left(y^{1-2s} \nabla E_R(\psi) \right) = 0 \qquad \text{in } \mathcal C_R \, ,
\]
we deduce that
\begin{equation}\label{eq234f}
\begin{aligned}
\int_{\mathcal C_{R, \varepsilon}} E(\Phi(u)) \dive\!\left[ y^{1-2s} \, \nabla \!\left(E_R(\psi) \, \xi_k\right) \right] \dx \dy = & \, 2 \int_{\mathcal C_{R, \varepsilon}} E(\Phi(u)) \left\langle \nabla E_R(\psi) , \nabla \xi_k \right\rangle  y^{1-2s}  \dx \dy  \\
 & + \int_{\mathcal C_{R, \varepsilon}} E(\Phi(u)) \, E_R(\psi) \dive\!\left( y^{1-2s} \, \nabla \xi_k \right) \dx \dy \,.
\end{aligned}
\end{equation}
By virtue of \eqref{cutoff-y} and \eqref{cutoff-y-k}, it is easy to check that
\begin{equation}\label{prop-cutoff}
\left|\nabla \xi_k \right| \leq \frac{C}{y} \, \chi_{[k, k+1]} \, , \qquad \left|\dive\!\left( y^{1-2s} \, \nabla \xi_k \right)\right|\leq \frac{C}{y^{1+2s}} \, \chi_{[k,k+1]} \, ,
\end{equation}
where $C>0$ is a suitable constant independent of $k$. Hence, thanks to \eqref{eq234f}--\eqref{prop-cutoff}, \eqref{stima-L2}--\eqref{stima-Hs} and the uniform boundedness of $ E(\Phi(u)) $, there holds
\begin{equation*}
\int_{\mathcal C_{R, \varepsilon}} E(\Phi(u)) \dive\!\left[ y^{1-2s} \, \nabla \!\left(E_R(\psi)\,\xi_k\right) \right] \dx \dy \xrightarrow[k\to\infty]{} 0 \, .
\end{equation*}
Besides, property \eqref{ass-1} and \eqref{stima-L2}--\eqref{stima-integrale} ensure that
\begin{equation*}
\left\langle \nabla E(\Phi(u)), \nabla E_R(\psi) \right\rangle , \, \nabla E(\Phi(u)) E_R(\psi) \in L^1\!\left(\mathcal C_R ; y^{1-2s} \dx\dy \right) ,
\end{equation*}
whence, exploiting again \eqref{prop-cutoff}, we can assert that
\begin{equation}\label{eq237f}
\int_{\mathcal C_{R, \varepsilon}} \left \langle \nabla E(\Phi(u)) , \nabla \! \left(E_R(\psi) \, \xi_k \right) \right \rangle y^{1-2s} \dx \dy \xrightarrow[k\to\infty]{} \int_{\mathcal C_{R, \varepsilon}} \left \langle \nabla E(\Phi(u)) , \nabla E_R(\psi) \right \rangle y^{1-2s} \dx \dy \, .
\end{equation}
Thus, from \eqref{eq235f}--\eqref{eq237f} we obtain the inequality
\begin{equation}\label{eq238f}
\begin{aligned}
\int_{\mathcal C_{R, \varepsilon}} \left \langle \nabla E(\Phi(u)) , \nabla E_R(\psi) \right \rangle y^{1-2s} \dx \dy
\leq - \int_{B_R\times\{\varepsilon\}} E(\Phi(u)) \, y^{1-2s} \, \frac{\partial E_R(\psi)}{\partial y}\, \dx \,.
\end{aligned}
\end{equation}
In view of Lemmas \ref{ext-prop} and \ref{lemma-nochetto}, by letting $\varepsilon\downarrow 0 $ in \eqref{eq238f} we end up with
\begin{equation}\label{eq239f}
\begin{aligned}
\mu_s \int_{\mathcal C_{R}} \left \langle \nabla E(\Phi(u)) , \nabla E_R(\psi) \right \rangle y^{1-2s} \dx \dy
\leq \int_{B_R} \Phi(u) \, (-\Delta)^s_R \, \psi \, \dx \,.
\end{aligned}
\end{equation}
A time integration of \eqref{eq239f} and \eqref{eq230f} applied to $ \varphi \equiv \psi $ yield \eqref{eq18f2}.

Let us finally remove the extra assumption \eqref{ass-1}. Given $h>0$, to any $f\in L^1_{\mathrm{loc}}([0, \infty))$ we associate its {\it Steklov average}, defined as
\[
f_h(t):=\frac 1 h\int_{t}^{t+h}f(s) \, \ds \,.
\]
It is not difficult to check that the Steklov averages of $u$ and $\Phi(u)$, that is $u_h $ and $ (\Phi(u))_h$, satisfy the very weak formulation
\begin{equation*}
\begin{aligned}
\int_0^T\int_{\mathbb{R}^d} u_h \, \varphi_t \,\dx\dt = \, & \int_0^T\int_{\mathbb{R}^d}(\Phi(u))_h \, (-\Delta)^s\varphi \, \dx\dt \\
&+\int_{\mathbb{R}^d}u_h(x,T)\,\varphi(x,T)\,\dx - \int_{\mathbb{R}^d}u_h(x, 0)\,\varphi(x,0)\,\dx
\end{aligned}
\end{equation*}
for a.e.~$T>0$, for every $\varphi\in C^\infty_c(\mathbb R^d\times[0,T])$. Moreover, by construction, it is plain that
\begin{equation}\label{eq256}
(\Phi(u))_h\in L^\infty(\mathbb R^d\times (0, \infty))
\end{equation}
and
\begin{equation*}
(-\Delta)^s (\Phi(u))_h(\cdot,t) = \frac{u(\cdot,t)-u(\cdot,t+h)}{h} \qquad \text{in } \mathcal D^\prime(\mathbb R^d) \, , \ \text{for a.e. } t>0 \, .
\end{equation*}
Since $u\in L^\infty(\mathbb R^d\times (0, +\infty))$, in view of \eqref{eq256} we are allowed to apply Lemma \ref{lem7}, which in particular ensures that
\begin{equation}\label{est-energy-um}
\begin{aligned}
& \int_{\mathcal{C}_R} \left| \nabla E((\Phi(u))_h)(x,y,t) \right|^2 \rho_{\lambda_1}(y) \, y^{1-2s}  \, \dx \dy \\
\le & \, \frac{2}{\mu_s} \int_{B_{2R}} \left| (\Phi(u))_h(x,t) \, (-\Delta)^s (\Phi(u))_h(x,t) \right| \dx +  4C^2 \, \| \Phi(u) \|_\infty^{2} \left|B_{2R}\right| \int_0^{\infty} \rho_{\lambda_1}(y) \, y^{1-2s}  \, \dy \, ,
\end{aligned}
\end{equation}
for a.e.~$ t>0 $. A time integration of \eqref{est-energy-um} then shows that \eqref{ass-1} is satisfied upon replacing $ \Phi(u) $ with $ (\Phi(u))_h $. Hence, by repeating the same arguments as above with $ u_h $ in place of $u$ and $ (\Phi(u))_h $ in place of $\Phi(u)$, we can deduce the validity of
\begin{equation}\label{eq257}
\begin{aligned}
\int_0^T \int_{B_R} u_h \, \psi_t \, \dx \dt \le & \, \int_0^T \int_{B_R} (\Phi(u))_h \, (-\Delta)^s_R\,\psi \,\dx \dt \\
& \, + \int_{B_R} u_h(x,T)\,\psi(x,T)\,\dx-\int_{B_R}u_h(x, 0)\,\psi(x,0)\,\dx
\end{aligned}
\end{equation}
for a.e.~$ T>0 $ and any nonnegative $\psi \in C^1\!\left([0, T]; L^2(B_R)\right)\cap L^2\!\left((0, T);\operatorname{Dom}(-\Delta)_R^s\right)$. The thesis then follows upon letting $ h \downarrow 0$ in \eqref{eq257} and exploiting standard convergence properties of the Steklov averages.
\end{proof}

\subsection{Proof of Proposition \ref{comppr}: the duality method}

The proof is based on a well-established technique known in the literature as \emph{duality method} (see \cite{ACP}). The basic idea consists in picking special test functions in \eqref{eq18f} and \eqref{eq18f2}, which formally are solutions to the following \emph{backward} parabolic problems:
\begin{equation}\label{eq15fb}
\begin{cases}
\psi_t = a(x,t) \, (-\Delta)^s_R \, \psi & \text{in } B_R\times (0,T) \, , \\
\psi = 0 & \text{on } \partial B_R \times (0, T) \, ,\\
\psi = \chi & \text{on } B_R \times \{ T \}\,,
\end{cases}
\end{equation}
where $ \chi $ is an arbitrary (sufficiently regular) final datum and $ T>0 $ is a free parameter. The coefficient $ a $ is defined by
\begin{equation}\label{adef}
a(x,t):=\begin{cases}\frac{\Phi(u(x,t))-\Phi(w(x,t))}{u(x,t)-w(x,t)}&\text{if}\ u(x,t)\not=w(x,t) \, , \\
0&\text{otherwise} \, ;
\end{cases}
\end{equation}
for notational convenience, here and hereafter we set $w\equiv u_R$. Note that $u, w$ being bounded and $u\mapsto \Phi(u)$ being locally Lipschitz, $a$ is also bounded; moreover, it is nonnegative $ \Phi $ being nondecreasing. However, in general the existence of a sufficiently regular solution to \eqref{eq15fb} is not guaranteed, so that one has to deal with suitable approximations. Since we rely on (parabolic) semigroup theory, in the corresponding approximating problems in place of $a(x,t)$ we will consider a sequence of functions which are regular, bounded away from zero, piecewise constant in time and suitably converge to $a$.

First of all, we need the following standard result.
In the sequel, by \it mild solution \rm we will mean a solution in the sense of semigroups, see e.g.~the classical reference \cite{P}.
\begin{pro}\label{dual}
Let $R \ge 1 $ and $\beta\in C^{\infty}\!\left(\overline B_R\right)$, with $\beta>0$. Let $\chi\in C_c^\infty(B_R)$, with $0\le \chi\le1$. Then the backward parabolic problem
\begin{equation*}
\begin{cases}
\psi_t = \beta \, (-\Delta)^s_R  \, \psi & \text{in } B_R\times (0, T) \, , \\
\psi = 0 & \text{on } \partial B_R \times (0, T) \, , \\
\psi = \chi & \text{on } B_R \times \{T\} \, ,
\end{cases}
\end{equation*}
admits a unique mild solution. Moreover, $\psi \in C^1\left([0, T]; L^2(B_R)\right) \cap L^2\left((0, T); \operatorname{Dom}(-\Delta)^s_R \right)$ and the following {\em energy estimate} holds:
\begin{equation*}
\int_0^T\int_{B_R} \beta \left[(-\Delta)^s_R \, \psi\right]^2 \dx\dt + \frac 1 2\int_{B_R}\left[(-\Delta)_R^{\frac s 2} \, \psi(x, 0)\right]^2 \dx = \frac 1 2\int_{B_R}\left[(-\Delta)_R^{\frac s 2} \, \chi \right]^2 \dx \, .
\end{equation*}
\end{pro}

We are now in position to prove Proposition \ref{comppr}.
\begin{proof}[Proof of Proposition \ref{comppr}]
Since $u$ and $ w\equiv u_R$ are bounded by definition and $ \Phi $ is locally Lipschitz, it is plain that $a$ is bounded as well; moreover, it is nonnegative because $\Phi $ is nondecreasing. Then, as a consequence of \eqref{eq18f} and \eqref{eq18f2}, for a.e.~$ T>0 $ and any nonnegative $\psi \in C^1\left([0, T]; L^2(B_R)\right) \cap L^2\left((0, T);\operatorname{Dom}(-\Delta)_R^s \right)$ (as in Proposition \ref{ineq}) there holds
\begin{equation}\label{eq20}
\int_0^T\int_{B_R}(u-w)\left[-a \, (-\Delta)_R^s \, \psi + \psi_t\right] \dx\dt
\le \, \int_{B_R}\left[u(x,T)-w(x,T)\right]\psi(x,T) \, \dx \, .
\end{equation}
Take a sequence of smooth functions $\{a_k\}_{k\in\mathbb{N}}$ in $\overline{B}_R$ such that $\tfrac 1k \le a_k \le \tfrac 1k+\|a\|_\infty$ and
\begin{equation}\label{eq27}
\displaystyle
\frac{a_k-a}{\sqrt{a_k}} \xrightarrow[k\to\infty]{} 0 \qquad \text{in } L^2(B_R\times(0,T)) \, .
\end{equation}
It is not difficult to show that such an approximating sequence does exist, see again \cite{ACP}. Now let $n\in\mathbb N$ and $h=0,\ldots,n$. Put $T_h:=\frac hn \, T$. Finally, for any fixed $k$, let $\{a_{n,k}\}_{n\in\mathbb N}$ be a sequence of functions which are constant in time, and smooth in $x$, in every subinterval $(T_h,T_{h+1})$ and converge to $a_k$ as follows:
\begin{equation}\label{bounds}
a_{n,k} \xrightarrow[n\to\infty]{} a_k \quad \text{pointwise a.e.}, \qquad \frac1{2k}\le a_{n,k}\le\frac{2}k+\|a\|_\infty \quad \forall k,n \in \mathbb{N} \,.
\end{equation}
Take any $\chi\in C^\infty_c(B_R)$ with $0\le\chi\le1$. For every $n, k\in\mathbb{N}$ and $ h\in \{0, \ldots, n-1\}$, let $\psi_{h}$ be recursively defined as the (mild) solution to
\begin{equation}\label{eq15bf}
\begin{cases}
\psi_t = a_{n,k} \, (-\Delta)^s_R \, \psi & \text{in } B_R \times (T_h, T_{h+1}) \, ,\\
\psi = 0 & \text{on } \partial B_R \times (T_h, T_{h+1}) \, ,\\
\psi = \psi_{h+1} & \text{on } B_R \times \{T_{h+1}\}\, ,
\end{cases}
\end{equation}
where we set $\psi_n=\chi$. Note that such a solution exists by virtue of Proposition \ref{dual}. Moreover, since the semigroup associated to the operator $ A := a_{n,k} \, (-\Delta)^s_R$ in the space
$L^2\!\left(B_R;a_{n,k}^{-1} \lfloor_{( T_h,T_{h+1})} \right)$
is Markov (see \cite{BD,P}), we can deduce that
\begin{equation*}
0\le \psi_{h}\le 1 \qquad \text{in }B_R\times(T_h,{T_{h+1}}) \, .
\end{equation*}
Besides, Proposition \ref{dual} ensures that for every $h=0, \ldots, n-1$ there holds
\begin{equation}\label{eq220f}
\begin{aligned}
& \int_{T_h}^{T_{h+1}} \int_{B_R} a_{n,k}\left[(-\Delta)^s_R \psi_h\right]^2 \dx\dt + \frac 1 2\int_{B_R}\left[(-\Delta)_R^{\frac s 2} \, \psi_h(x, T_h)\right]^2 \dx \,\\
 = & \, \frac 1 2\int_{B_R}\left[(-\Delta)_R^{\frac s 2} \, \psi_h(x, T_{h+1})\right]^2 \dx\,.
\end{aligned}
\end{equation}
Put
$$
\psi_{n, k}(x,t):=\psi_{h}(x,t) \quad \forall (x,t) \in B_R \times [T_h, T_{h+1}] \, , \qquad \text{for every}\;\; h=0, \ldots, n-1 \, ;
$$
note that $\psi_{n,k}$ is a well-defined and continuous function in $B_R\times [0, T]$. Summing up both sides of \eqref{eq220f} from $h=0$ to $h=n-1$, we end up with the identity
\begin{equation}\label{eq221f}
\int_{0}^{T}\int_{B_R} a_{n,k} \left[(-\Delta)^s_R \, \psi_{n,k}\right]^2 \dx\dt + \frac 1 2\int_{B_R}\left[(-\Delta)_R^{\frac s 2} \, \psi_{n,k}(x, 0)\right]^2 \dx = \underbrace{\frac 1 2\int_{B_R}\left[(-\Delta)_R^{\frac s 2} \, \chi \right]^2 \dx}_{C_R^2} .
\end{equation}
Hence, from \eqref{eq20} with $ \psi = \psi_{n,k}$ we obtain
\begin{equation}\label{eq23}
\begin{aligned}
 & \int_0^T \int_{B_R} (u-w)\left[-(a-a_{n,k}+a_{n,k})\,(-\Delta)_R^s \, \psi_{n,k} + \left( \psi_{n,k} \right)_t \right] \dx\dt \\
\le & \, \int_{B_R}[u(x,T)-w(x,T)] \, \chi(x) \, \dx \,.
\end{aligned}
\end{equation}
On the other hand, thanks to \eqref{eq15bf}, from \eqref{eq23} there follows
\begin{equation}\label{eq25}
\begin{aligned}
\int_0^T \int_{B_R} (u-w) \, (a_{n,k}-a) \, (-\Delta)_R^s \, \psi_{n,k} \, \dx \dt \le \int_{B_R}[u(x, T) - w(x, T)]\, \chi(x) \, \dx \, .
\end{aligned}
\end{equation}
In addition, the uniform (w.r.t.~$n,k$) estimate \eqref{eq221f}, \eqref{eq27} and \eqref{bounds} yield
\begin{equation}\label{eq26}
\begin{aligned}
&\int_0^T\int_{B_R}|u-w|\,|a-a_{n,k}|\left|(-\Delta)_R^s \, \psi_{n,k} \right| \dx \dt \\
\leq \, & (\|u \|_\infty+ \|w\|_\infty) \int_0^T \int_{B_R} \frac{|a_{n,k}-a|}{\sqrt{a_{n,k}}} \sqrt{a_{n,k}} \left|(-\Delta)_R^s \, \psi_{n,k} \right|  \dx \dt \\
\leq \, & (\|u \|_\infty+ \|w\|_\infty) \left\|\frac{a-a_{n,k}}{\sqrt{a_{n,k}}} \right\|_{L^2(B_R \times (0,T))} C_R \, .
\end{aligned}
\end{equation}
By collecting \eqref{eq25}--\eqref{eq26}, letting first $n\to \infty$ and then $k\to \infty$ (recalling \eqref{eq27}--\eqref{bounds}), we finally infer that
\[\int_{B_R} \left[u(x, T)- w(x,T)\right] \chi(x) \, \dx \geq 0 \, , \]
whence $ u \ge w $ a.e.~in $ B_R \times \mathbb{R}^+ $ in view of the arbitrariness of $ T $ and $ \chi $.
\end{proof}

\section{Existence and uniqueness: proofs}\label{proofs}

Since it is more natural to our strategy, we will first address the case of nonnegative solutions and $\Phi$ satisfying $ \Phi(0)=0 $  in Subsection \ref{nonneg}, for which we can exploit the tools of Section \ref{mini}, and then consider general solutions and nonlinearities $ \Phi $ in Subsection \ref{sign-change}, upon adapting the previous arguments.

\subsection{Nonnegative bounded solutions}\label{nonneg}

Prior to proving existence of the minimal solution (Proposition \ref{thmesi}), we need to ensure the well-posedness of problem \eqref{eq17f} along with some crucial comparison properties.
\begin{pro}\label{ex-R}
Let $ s \in (0,1) $, $R \ge 1 $, $ u_0 \in L^\infty(B_R) $ with $ u_0 \ge 0 $ and $ \Phi $ satisfy \eqref{hp-phi}  with $\Phi(0)=0$. Then there exists a solution $u_R$ to problem \eqref{eq17f} in the sense of Definition \ref{D1f}. More precisely, for a.e.~$T>0$ there holds
\begin{equation}\label{eq400}
\begin{aligned}
\int_0^T \int_{B_R} u_R \, \psi_t \, \dx \dt = & \, -  \int_0^T\int_{B_R} (-\Delta)^{s/2}_R \, \Phi(u_R) \, (-\Delta)^{s/2}_R \, \psi \, \dx \dt \\
& \, + \int_{B_R} u_R(x,T) \, \psi(x,T)\,\dx - \int_{B_R}u_0(x) \, \psi(x,0) \,\dx
\end{aligned}
\end{equation}
for any $\psi \in C^1\!\left([0, T]; L^2(B_R)\right) \cap L^2\!\left((0, T);H^s_0(B_R)\right)$. Moreover,
\begin{equation}\label{ubb}
\Phi(u_R)\in L^2\!\left((0, T);H^s_0(B_R) \right), \qquad 0\leq u_R \leq \|u_0\|_\infty \, ,
\end{equation}
and the local energy estimate
\begin{equation}\label{eq410}
\int_0^T\int_{\Omega_r}\left|\nabla E_R(\Phi(u_R))\right|^2 y^{1-2s} \dx \dy \dt\leq \frac{2}{\mu_s} \int_{\Omega_{2r}} \Psi(u_0)\, \dx  + C \, \|\Phi(u_0)\|_{\infty}^2\int_{\Omega_{2r}} y^{1-2s}\, \dx\dy
\end{equation}
is valid for all $r\in (1/4, R/2)$, for some $C>0$ independent of $r, R$, where $\Psi(u):=\int_0^u\Phi(v) \, \dv$. If in addition $u_R$ and $w_R$ are solutions to \eqref{eq17f} starting from the ordered initial data $u_0 \le w_0$, respectively, then $u_R\leq w_R$. Finally, for every $ 1 \le R_1 \leq R_2$ there holds
\begin{equation}\label{eqqq}
u_{R_1} \leq u_{R_2} \qquad \text{a.e.~in } B_{R_1}\times \mathbb{R}^+ \, ,
\end{equation}
where by $ u_{R_1} $ and $ u_{R_2} $ we denote \emph{any} two solutions to \eqref{eq17f}, in the sense of Definition \ref{D1f}, corresponding to the same initial datum $ u_0 \in L^\infty(B_{R_2}) $ with $ u_0 \ge 0 $. In particular, the solution to problem \eqref{eq17f} in the sense of Definition \ref{D1f} is unique.
\end{pro}
\begin{proof}
The most used technique in the literature to construct energy solutions to \eqref{eq17f}, i.e.~solutions satisfying \eqref{eq400} and \eqref{ubb}, or to similar problems, relies on the celebrated Crandall-Liggett Theorem, which goes back to \cite{CL}. The basic strategy (see \cite[Chapter 10]{Vbook} in the local case) consists in first solving the discretized ``resolvent'' equation
\begin{equation}\label{resolvent}
\frac{u_{n+1}-u_n}{\tau} = -\left( -\Delta \right)^s \Phi(u_{n+1})  \qquad \text{in } B_R \, , \quad \forall n \in \mathbb{N} \, ,
\end{equation}
where $ \tau>0 $ is a fixed time step, and then suitably letting $\tau \downarrow 0$, thus obtaining a solution to \eqref{eq17f} as a limit of the piecewise-constant interpolants (in time) of the sequences $ \{ u_n \}_{n \in \mathbb{N}} $. Let us point out that, in order to solve \eqref{resolvent} at each fixed $ \tau>0 $, one may further approximate $ \Phi $ with a sequence $ \{ \Phi_k \}_{k \in \mathbb{N}} $ of regular, strictly increasing and non-degenerate nonlinearities. The order-preserving property $ u_R \le w_R $ is also a consequence of such a construction. Since the procedure is by now rather standard, we will not give further details: we refer to \cite[Theorem 7.2]{Vaz12} for the porous-medium case ($ \Phi(u)=u^m $ with $ m>1 $) and to  \cite[Theorem 3.7]{AMS} for diffusion-type equations governed by a wide class of operators in abstract frameworks.

Once we have at our disposal an energy solution $u_R$, the validity of \eqref{eq18f} is a simple consequence of \eqref{eq400} along with \eqref{eq401}. Besides, the local energy estimate \eqref{eq410} can formally be proved by minor variations to the proof of Lemma \ref{lem7} (see Appendix \ref{App3}): the idea is to test the ``extended'' version of \eqref{eq400} with the function $ E_R\!\left(\Phi(u_R)\right)\! \eta_r $, where the cut-off $\eta_r$ is defined as in \eqref{cutoff-k} with the additional constraint $| \nabla \eta_r |^2 \le C \eta_r $. However, since a priori $ \left( \Phi(u_R) \right)_t $ may not make sense, in order to establish it rigorously one can obtain analogous discrete estimates on \eqref{resolvent}, sum up in $n$ and then let $ \tau \downarrow 0 $.

To conclude the proof, we need to show \eqref{eqqq}. We claim that
\begin{equation}\label{eq402}
\begin{aligned}
\int_0^T \int_{B_{R_1}} u_{R_2} \, \psi_t \, \dx \dt \le & \, \int_0^T \int_{B_{R_1}} \Phi(u_{R_2}) \, (-\Delta)^s_{R_1} \psi \,\dx \dt \\
& + \int_{B_{R_1}} u_{R_2}(x,T) \, \psi(x,T)\,\dx-\int_{B_{R_1}}u_0(x) \, \psi(x,0)\,\dx
\end{aligned}
\end{equation}
for any nonnegative $\psi \in C^1\!\left([0, T]; L^2(B_{R_1})\right)\cap L^2\!\left((0, T);\operatorname{Dom}(-\Delta)_{R_1}^s \right)$. Indeed, the identity
\begin{equation}\label{eq404}
\int_0^T \int_{B_{R_2}} \Phi(u_{R_2}) \,(-\Delta)_{R_2}^s\,\varphi \, \dx \dt  = \mu_s \int_0^T\int_{\mathcal C_{R_1}} \left \langle \nabla E_{R_2}(\Phi(u_{R_2})) , \nabla E_{R_1}(\varphi) \right \rangle y^{1-2s} \dx \dy \dt
\end{equation}
holds for every $ \varphi \in C^\infty_c(B_{R_1}\times [0, T])$. Hence \eqref{eq402} follows similarly to the proof of Proposition \ref{ineq}, upon replacing the integral version of \eqref{eq403} with \eqref{eq404}. Note that here we need not use Steklov averages, since since both $u_{R_1}$ and $u_{R_2}$ are energy solutions. Thanks to \eqref{eq402}, by arguing as in the proof of Proposition \ref{comppr}, replacing $u$ and $w$ with $u_{R_2}$ and $u_{R_1}$, respectively, we infer \eqref{eqqq}. We finally mention that uniqueness follows from \eqref{eqqq} with $ R_1=R_2=R $.
\end{proof}

\begin{proof}[Proof of Proposition \ref{thmesi}] The proof essentially relies on the existence and the comparison principles established above. Indeed, let $k \in \mathbb{N}$, with $k\geq 1$, and put $u_{0k}:=u_0\chi_{B_k}$. It is apparent that $u_{0k}\in L^1(\mathbb R^d)\cap L^{\infty}(\mathbb R^d)$.
First one solves problem \eqref{eq17f} on $B_R$ (let $ R\ge1 $), which does have a unique (energy) solution $u_{k, R}$ thanks to Proposition \ref{ex-R}. Then, still in view of Proposition \ref{ex-R}, the family of solutions $\{u_{k, R}\}_{R \ge 1}$ is monotone increasing with respect to $ R $ (recall \eqref{eqqq}), uniformly bounded by virtue of \eqref{ubb} and for any $1\leq k_1\leq k_2$, $R \ge 1$, there holds $u_{k_1, R}\leq u_{k_2, R}$ a.e.~in $B_R \times \mathbb{R}^+ $ (consequence of the first comparison property). Hence, for any fixed $k\geq 1$, we can assert that there exists $u_k:=\lim_{R\to\infty} u_{k,R}$ and it satisfies $0\leq u_k\leq \|u_0\|_\infty$ in $\mathbb R^d\times
\mathbb{R}^+$. Moreover, $u_{k_1}\leq u_{k_2}$ a.e.~in $ \mathbb{R}^d \times \mathbb{R}^+$. Since $u_{0k}\in L^1(\mathbb R^d)\cap L^{\infty}(\mathbb R^d)$, by arguing as in \cite[Section 7.2]{Vaz12} it is not difficult to show that $u_k$ is a weak solution to problem \eqref{eq6} with initial datum $ u_{0k} $, in the sense that
\begin{equation}\label{eq406}
\begin{aligned}
\int_0^T\int_{\mathbb{R}^d} u_k \, \varphi_t
\,\dx\dt = & \int_0^T \int_{\mathbb{R}^d}(-\Delta)^{{s}/{2}} \Phi(u_k) \, (-\Delta)^{{s}/{2}}\varphi\,\dx\dt\\
&+\int_{\mathbb{R}^d} u_k(x,T) \, \varphi(x,T)\,\dx-\int_{\mathbb{R}^d} u_{0k}(x)\,\varphi(x,0)\,\dx
\end{aligned}
\end{equation}
for every $\varphi\in C_c^{\infty}(\mathbb{R}^d\times[0,T])$. Since the sequence $\{u_k\}_{k\geq 1}$ is monotone increasing w.r.t.~$k\in \mathbb N$ and uniformly bounded, there exists $\underline u:=\lim_{k\to \infty} u_k $ and it satisfies $0\leq \underline{u} \leq \|u_0\|_\infty$. Moreover, integrating by parts and then passing to the limit as $k\to \infty$ in \eqref{eq406} we easily obtain  \eqref{eq7} (with $ u=\underline{u} $). The fact that $\underline u$ is indeed minimal follows straightly from Proposition \ref{comppr}.
\end{proof}

We will now prove Theorem \ref{teouni} in the case of nonnegative solutions, upon exploiting the existence of the minimal one. In what follows, we take for granted the cut-off functions $ \gamma_R $ defined in \eqref{eq:cutoff}, since they will be used several times.

\begin{proof}[Proof of Theorem \ref{teouni} (nonnegative solutions and $\Phi(0)=0$)]
In addition to the hypotheses of Theorem \ref{thmesi-sign-ch}, we also assume that $ u_0 \ge 0 $, $\Phi(0)=0$ and that $ u $ is any \emph{nonnegative} solution. Proposition \ref{thmesi} guarantees the existence of the minimal solution $ \underline{u} $, so that $ u $ is necessarily larger than $ \underline{u} $. To our purposes, let us consider a regular and radially-decreasing function $ h(x) \equiv h(|x|) $ such that
\begin{equation}\label{eq-h}
\frac{c_1}{1+|x|^{\alpha}} \le h(x) \le \frac{c_2}{1+|x|^\alpha} \qquad \text{and} \qquad \left| \nabla^2 h(x) \right| \le \frac{c_2}{1+|x|^{\alpha+2}} \qquad \forall x \in \mathbb{R}^d \, ,
\end{equation}
for some $ \alpha \in (d,d+2s) $ and positive constants $c_1,c_2$. Under assumption \eqref{eq-h} one can prove (see e.g.~\cite[Lemma 2.1]{BVadv}) that $ (-\Delta)^s h(x) $ is also a regular function satisfying
\begin{equation}\label{eq-h-2}
\left| (-\Delta)^s h(x) \right| \le \frac{c^\prime}{1+|x|^{d+2s}} \qquad \forall x \in \mathbb{R}^d \, ,
\end{equation}
for a suitable positive constant $ c^\prime $. By formula \eqref{def-frac}, it is easy to check that
\begin{equation*}\label{lap-prod}
(-\Delta)^s\!\left(h\gamma_R\right) = h \, (-\Delta)^s\gamma_R + (-\Delta)^s h \, \gamma_R + 2 \mathcal Q(h, \gamma_R)\, ,
\end{equation*}
where the quadratic form $ \mathcal{Q} $ is defined as
$$
\mathcal Q(h, \gamma_R)(x):=c_{d,s}\int_{\mathbb R^d}\frac{[h(x)-h(y)][\gamma_R(x)-\gamma_R(y)]}{|x-y|^{d+2s}}\, \dy \qquad \forall x\in \mathbb R^d\,.
$$
Hence, from Definition \ref{defsol2} with the choice $ \varphi(x,t) \equiv \varphi(x)=h(x)\gamma_R(x)$, we infer that for a.e.~$ T>0 $ there holds
\begin{equation}\label{eq39}
\begin{aligned}
& \int_0^T\int_{\mathbb R^d}\left[\Phi(u(x,t))-\Phi(\underline u(x,t))\right] (-\Delta)^s h(x) \, \gamma_R(x) \, \dx \dt \\
& +\int_{\mathbb R^d}\left[u(x, T)-\underline u(x, T)\right] h(x) \, \gamma_R(x) \, \dx\\
= & -\int_0^T\int_{\mathbb R^d}\left[ \Phi(u(x,t))-\Phi(\underline u(x,t))\right] \left[h(x)\,(-\Delta)^s\gamma_R(x) + 2\mathcal Q(h, \gamma_R)(x)\right] \dx \dt =: \mathcal I(R) \, .
\end{aligned}
\end{equation}
We aim at showing that
\begin{equation}\label{eq41}
\lim_{R\to\infty}\mathcal I(R)=0\,.
\end{equation}
Indeed, as recalled e.g.~in \cite[Lemmas 3.2, 3.3]{Mur}, the $s$-fractional Laplacian of cut-off functions satisfies
\begin{equation}\label{eq261-bis}
\left|(-\Delta)^s \gamma_R(x) \right| \leq \frac{C}{R^{2s}} \, \frac{1}{1+\left(\frac{|x|}{R} \right)^{d+2s}} \qquad \forall x\in \mathbb R^d \, ,
\end{equation}
where from here on $C$ is a generic positive constant independent of $R \ge 1$. Clearly \eqref{eq261-bis} is enough to deduce that
\begin{equation}\label{eee}
\lim_{R \to \infty} \int_0^T\int_{\mathbb R^d}\left[ \Phi(u(x,t))-\Phi(\underline u(x,t)) \right] h(x)\,(-\Delta)^s\gamma_R(x) \, \dx \dt = 0 \, ,
\end{equation}
given the boundedness of $ u $ and $ \underline{u} $ along with the fact that $ h \in L^1(\mathbb{R}^d) $. The second term in the r.h.s.~of \eqref{eq39} must be handled more carefully. First of all, for any $p\in (1, \infty)$ (and a regular enough function $f$) we introduce the nonlinear operator
\[
\mathcal T_p(f)(x):=\int_{\mathbb R^d}\frac{|f(x)-f(y)|^p}{|x-y|^{d+ps}} \, \dy \qquad \forall x \in \mathbb{R}^d \, .
\]
Thanks to \cite[Lemmas 3.2, 3.3]{Mur}, an estimate analogous to \eqref{eq261-bis} holds:
\begin{equation}\label{eq310}
\mathcal T_p(\gamma_R)(x) \leq \frac{C}{R^{ps}} \, \frac{1}{1+\left(\frac{|x|}{R} \right)^{d+ps}} \qquad \forall x \in \mathbb R^d\,.
\end{equation}
On the other hand, by reasoning similarly to the proof of \cite[Lemma 5.2]{Mur}, it is not difficult to show that
\begin{equation}\label{eq310-bis}
\mathcal T_p(h)(x) \leq \frac{C}{1+|x|^{d+ps}} \qquad \forall x \in \mathbb R^d\,.
\end{equation}
By combining \eqref{eq310}--\eqref{eq310-bis} and H\"older's inequality, we obtain (let $ 1/p + 1/p^\prime = 1 $)
\begin{equation}\label{eq311}
\left|\mathcal Q(h, \gamma_R)(x) \right| \leq c_{d,s} \, \mathcal T_p^{\frac 1 p}\!(h)(x) \, \mathcal T_{p'}^{\frac 1{p'}}\!(\gamma_R)(x) \leq \frac{C}{R^s} \, \frac1{\left(1+|x|^{\frac{d}{p}+s}\right)\left[1+\left(\frac{|x|}{R} \right)^{\frac{d}{p'}+s}\right]} \qquad \forall x\in \mathbb R^d\,.
\end{equation}
As a consequence of estimate \eqref{eq311}, an elementary computation yields
\begin{equation}\label{eq311-bis}
\int_{\mathbb{R}^d} \left|\mathcal Q(h, \gamma_R)(x) \right| \mathrm{d}x \leq \frac{C}{R^{2s-\frac{d}{p^\prime}}} \, ,
\end{equation}
up to a logarithmic correction in the critical case $ d=sp^\prime $. By virtue of \eqref{eq311-bis}, if we choose $ p $ so small that $ 2s>d/p^\prime $ then upon letting $ R \to \infty $ we deduce that
\begin{equation*}\label{QR}
\lim_{R \to \infty} \int_0^T\int_{\mathbb R^d}\left[ \Phi(u(x,t))-\Phi(\underline u(x,t)) \right] \left| Q(h, \gamma_R)(x) \right| \dx \dt = 0 \, ,
\end{equation*}
thanks again to the boundedness of $ u $ and $ \underline{u} $. This, together with \eqref{eee}, implies that \eqref{eq41} does hold, so that by passing to the limit in \eqref{eq39} as $ R \to \infty $ (exploiting also the fact that $ h , (-\Delta)^s h \in L^1(\mathbb{R}^d) $) we end up with the identity
\begin{equation}\label{diff-eq-1}
\int_{\mathbb R^d}\left[u(x, T)-\underline u(x, T)\right] h(x) \, \dx =
\int_0^T\int_{\mathbb R^d}\left[\Phi(\underline{u}(x,t))-\Phi(u(x,t))\right] (-\Delta)^s h(x) \, \dx \dt \, ,
\end{equation}
valid for a.e.~$ T> 0 $. Since $\Phi$ is locally Lipschitz, $ u $ and $ \underline{u} $ are bounded, estimates \eqref{eq-h}--\eqref{eq-h-2} hold and $ \alpha < d+2s $, from \eqref{diff-eq-1} there follows
\begin{equation}\label{diff-eq-2}
\int_{\mathbb R^d}\left[u(x, T)-\underline u(x, T)\right] h(x) \, \dx \le
C \int_0^T\int_{\mathbb R^d}\left[u(x, t)-\underline u(x, t)\right] h(x) \, \dx \dt \qquad \text{for a.e. } T>0 \, ,
\end{equation}
whence
\begin{equation*}\label{diff-eq-3}
\int_0^T\int_{\mathbb R^d}\left[u(x, t)-\underline u(x, t)\right] h(x) \, \dx \dt = 0 \qquad \forall T>0
\end{equation*}
by Gronwall's Lemma. This clearly implies $ u=\underline{u} $ because $ u \ge \underline{u} $ and $h$ is strictly positive.
\end{proof}

\subsection{General bounded solutions}\label{sign-change}

Our purpose is now to deal with initial data (and solutions) that may change sign and with nonlinearities $\Phi$ that are not forced to comply with $ \Phi(0)=0 $. To this end, the idea is to resort to the theory of nonnegative solutions that we have developed above, by means of the following simple observation: if $ u $ is a solution to \eqref{eq6}, then for any $ c \in \mathbb{R} $ the function $ \hat{u} := u-c $ is a solution to the same problem with initial datum $ \hat{u}_0 := u_0 -c $ and a slightly different nonlinearity.

\begin{pro}\label{translation}
Let $ s \in (0,1) $, $ u_0 \in L^\infty(\mathbb{R}^d) $ and $ \Phi $ satisfy \eqref{hp-phi}. Let $ c \in \mathbb{R} $. Then $u$ is a solution to problem \eqref{eq6}, according to Definition \ref{defsol2}, if and only if the function $ \hat{u} := u - c $ is a solution to the same problem with $u_0$ replaced by the initial datum $ \hat{u}_0 := u_0 - c $ and $ {\Phi} $ replaced by the nonlinearity
$$
\hat{\Phi}(v) := \Phi(v+c)-\Phi(c)  \qquad \forall v \in \mathbb{R} \, ,
$$
which still complies with \eqref{hp-phi} and in addition satisfies $ \hat{\Phi}(0)=0 $.
\end{pro}
\begin{proof}
It is just a matter of replacing $ u $ with $ \hat{u}+c $ in the very weak formulation \eqref{eq7} and integrating in time the left-hand side. The only nontrivial point consists in showing that the constant $ \Phi(c) $ can be added to the first integral term in the r.h.s.~of \eqref{eq7}. Nevertheless, this is a simple consequence of the identity
$$
\int_{\mathbb{R}^d} (-\Delta)^s \varphi(x) \, \mathrm{d}x = 0 \qquad \forall \varphi \in C^\infty_c(\mathbb{R}^d) \, ,
$$
which can be shown e.g.~by means of a standard cut-off argument starting from (we recall that $ \gamma_R $ is defined in \eqref{eq:cutoff})
$$
\int_{\mathbb{R}^d} \gamma_R(x) \, (-\Delta)^s \varphi(x) \, \mathrm{d}x = \int_{\mathbb{R}^d} (-\Delta)^s\gamma_R(x) \, \varphi(x) \, \mathrm{d}x \qquad \forall \varphi \in C^\infty_c(\mathbb{R}^d)
$$
and letting $ R \to \infty $, upon exploiting \eqref{eq261-bis}.
\end{proof}

\begin{proof}[Proof of Theorem \ref{thmesi-sign-ch}]
Let $ c := \operatorname{ess}\inf_{x \in \mathbb{R}^d} u_0(x) $ and set $ \hat{u}_0 = u_0 - c  \ge 0 $. By virtue of Proposition \ref{thmesi} we know that problem \eqref{eq6}  with $\Phi$ replaced by $\hat\Phi$ admits a (minimal) solution $ \hat{u} $ with initial datum $ \hat{u}_0 $. On the other hand, Proposition \ref{translation} ensures that $ u=\hat{u}+c $ is a solution to the original problem \eqref{eq6} with initial datum $u_0$.
\end{proof}

\begin{proof}[Proof of Theorem \ref{teouni} (sign-changing solutions and general $\Phi$)]
Let $ u_1 $ and $ u_2 $ be two possibly different solutions to \eqref{eq6}, starting from the same initial datum $ u_0 \in L^\infty(\mathbb{R}^d) $. Put
$$
c := \left( \underset{(x,t) \in \mathbb{R}^d \times \mathbb{R}^+ }{\operatorname{ess}\inf} \, u_1(x,t) \right) \wedge \left( \underset{(x,t) \in \mathbb{R}^d \times \mathbb{R}^+ }{\operatorname{ess}\inf} \, u_2(x,t) \right) ,
$$
which is finite since both $ u_1 $ and $ u_2 $ are bounded by assumption. If we set $ \hat{u}_1 := u_1 - c $ and $ \hat{u}_2 := u_2 - c $, still by Proposition \ref{translation} and by the definition of $c$ it is apparent that both $ \hat{u}_1 $ and $ \hat{u}_2 $ are nonnegative, bounded solutions to \eqref{eq6} with $\Phi$ replaced by $\hat\Phi$ and  initial datum $ \hat{u}_0 = u_0-c \ge 0 $. Hence, as a consequence of the first part of the proof of Theorem \ref{teouni} in the case nonnegative solutions, we can deduce that $ \hat{u}_1 = \hat{u}_2 $ and therefore $ u_1 = u_2 $.
\end{proof}

\begin{rem}\label{oss2}\rm
Due to Proposition \ref{ex-R}, for any $k\geq 1$ and $R\ge1$ the solution $u_{k, R}$ constructed in the proof of Proposition \ref{thmesi} satisfies \eqref{eq410}. Since such estimate is purely local, by passing to the limit first as $R\to \infty$ and then as $k\to \infty$, we infer that
\begin{equation}\label{eq407}
\nabla E(\Phi(\underline{u}))\in L^2_{\rm{loc}}\!\left(\overline{\Omega} \times [0,\infty) ; y^{1-2s} \mathrm{d}x\mathrm{d}y\mathrm{d}t \right) .
\end{equation}
Hence, in view of Theorem \ref{teouni}, it follows that \emph{any} nonnegative very weak solution to \eqref{eq6} is in fact a local weak (energy) solution, in the sense that \eqref{eq407} holds. Recalling the proof of Theorem \ref{thmesi-sign-ch} along with the fact that the (global) extension operator $ E $ is invariant with respect to addition of constants, the same is true for sign-changing solutions. Note that property \eqref{eq407}, in particular, implies that the function $ \Phi(\underline{u}) $ belongs to the space $  \dot{H}^s_{\mathrm{loc}}(\mathbb{R}^d) $, in the sense that
$$
\gamma_R \, \Phi(\underline{u}) \in L^2_{\mathrm{loc}}([0,\infty);\dot{H}^s(\mathbb{R}^d))
$$
for any cut-off function $ \gamma_R $ as in \eqref{eq:cutoff}, since $\gamma_R \, E (\Phi(\underline{u}))$ is a finite-energy extension of $\gamma_R \, \Phi(\underline{u})$. Again, the same holds for sign-changing solutions.
\end{rem}

\subsection{Final comments, open problems and generalizations}

We conclude our discussion by listing some technical points which are important for our methods of proof to work, along with some related open questions and possible generalizations.

\begin{enumerate}

\smallskip
\item The results of the present paper have been shown for nonlinearities $\Phi$ that are \emph{locally Lipschitz}. In the model case $\Phi(u)=u|u|^{m-1}$ this corresponds to the constraint $m\ge1$, hence we are not considering the \emph{fractional fast diffusion equation} (where $ m \in (0,1) $). Indeed, there are technical issues in dealing with the latter case, in particular as concerns the following two points. First, we use in a crucial way the \emph{boundedness} of the function $a$ given in \eqref{adef}, which fails when $m<1$. Secondly, the fact that $\Phi$ is Lipschitz is exploited in the passage from \eqref{diff-eq-1} to \eqref{diff-eq-2}, in order to end up with a closed integral inequality. The problem of proving an analogue of Theorem \ref{teouni} for the fractional fast diffusion equation is therefore open.

\smallskip
\item The problems considered in \cite{TJJ} involve a more general class of integral operators with singular kernels satisfying suitable conditions, which does include the fractional Laplacian. Our techniques, that proved to be successful in order to remove the extra assumption $u-w\in L^1({\mathbb R}^d\times(0,T))$, necessary for the methods of \cite{TJJ} to work, require that a suitable extension operator is well defined and that properties similar to the ones established in Sections \ref{sec:pre} and \ref{mini} hold. It is possible, for instance, that a careful adaptation of our strategy could allow for generalizations to nonlinear evolution equations driven by spectral fractional powers of uniformly elliptic operators, since for the latter certain extension results are available (see e.g.~\cite{ST}).

\smallskip
\item The study of inhomogeneous equations of the form
$$
u_t+(-\Delta)^s \, \Phi(u)=f(x,t) \qquad \text{in } \mathbb{R}^d \times \mathbb{R}^+
$$
 is not performed here. Nevertheless, one can check that, at least under the assumption
 $$
 f \in L^\infty_{\mathrm{loc}}([0,+\infty);L^\infty(\mathbb{R}^d)) \, ,
 $$
 our arguments still work with some modifications. Indeed, the most relevant change is due to the fact that the solutions to the analogues of the approximate problems \eqref{eq17f} are not necessarily positive. In particular, the monotonicity property \eqref{eqqq} is lost, but this is not an issue since a limit solution as $ R \to \infty $ can still be constructed by standard compactness tools of the weak$^*$ topology in $ L^\infty $. As a consequence, the ``minimal'' solution may also not be positive somewhere; however, this does not affect the overall strategy. Finally, solutions will in general belong to $ L^\infty_{\mathrm{loc}}([0,+\infty);L^\infty(\mathbb{R}^d)) $, i.e.~they will not necessarily be globally bounded in time. Again, no relevant additional difficulties arise, since one can work on bounded time intervals approximating $ \mathbb{R}^+ $.

    \end{enumerate}

\appendix
\section{Proofs of technical lemmas}\label{app}

In order to lighten the reading of the paper, we postpone to this appendix the rigorous proofs of some technical facts we exploit in Sections \ref{sec:pre} and \ref{mini}, which are of key importance.

\subsection{Proof of Lemma \ref{lemma-nochetto}}\label{Appnochetto}
Given $R>0$ and $ \varphi \in L^2(B_R) $, the following representation formula for $ E_R(\varphi) $ holds:
\begin{equation}\label{eq300}
E_R(\varphi)(x,y) = \sum_{k=1}^{+\infty} \hat\varphi_k \, \phi_k(x) \, \psi_k(y) \qquad \forall (x,y)\in \mathcal C_R \, ,
\end{equation}
where for every $ k \in \mathbb{N} $
\begin{equation}\label{eq301}
\psi_k(y) := c_s \left(\sqrt{\lambda_k} y \right)^s K_s\!\left( \sqrt{\lambda_k} y \right) \qquad \forall y>0
\end{equation}
with
\begin{equation}\label{ee11}
c_s := \frac{2^{1-s}}{\Gamma(s)}
\end{equation}
and $K_s$ denotes the modified \emph{Bessel function} of the second kind (see e.g.~\cite[Chapter 9.6]{Ab}), which enjoys some key properties that we will now recall, referring to \cite[Section 2.4]{Noch} and the literature quoted therein for more details. Indeed, $ z \in (0,\infty) \mapsto K_s (z) $ is positive, smooth and satisfies
\begin{equation}\label{eq302}
\lim_{z \downarrow 0}  c_s \, z^s K_s(z) = 1 \, ,
\end{equation}
\begin{equation}\label{eq303}
\frac{\mathrm{d}}{\dz}\left[ z^s K_s(z) \right] = -z^s \, K_{1-s}(z) \qquad \forall z>0 \, ,
\end{equation}
\begin{equation}\label{eq304}
z \mapsto z^{\min\left\{s, \frac 1 2\right\}} e^z K_s(z) \ \, \text{is a decreasing function in } (0, \infty)\,.
\end{equation}
\begin{proof}[Proof of Lemma \ref{lemma-nochetto}]
As concerns \eqref{stima-L2}, thanks to \eqref{eq300}--\eqref{eq301} we have:
\begin{equation}\label{L2-aaa}
\left\| E_R(\varphi) (\cdot,y) \right\|_{L^2(B_R)}^2 = \sum_{k=1}^{\infty} \hat{\varphi}_k^2 \, \psi_k^2(y) = c_s^2 \sum_{k=1}^{\infty} \hat{\varphi}_k^2 \left(\sqrt{\lambda_k} y \right)^{2s} K_s^2\!\left( \sqrt{\lambda_k} y \right) \qquad \forall y >0 \,;
\end{equation}
on the other hand, properties \eqref{eq302} and \eqref{eq304} ensure that for every $ \theta \in (0,1) $ there exists a constant $C_\theta $ as in the statement such that
\begin{equation*}
c_s \, z^s K_s(z) \le C_\theta \, e^{-\theta z} \qquad \forall z >0 \, ,
\end{equation*}
whence
\begin{equation*}
\begin{gathered}
c_s^2 \sum_{k=1}^{\infty} \hat{\varphi}_k^2 \left(\sqrt{\lambda_k} y \right)^{2s} K_s^2\!\left( \sqrt{\lambda_k} y \right) \le C_\theta^2 \sum_{k=1}^{\infty} \hat{\varphi}_k^2 \, e^{-2\theta\sqrt{\lambda_k}y} \le C_\theta^2 \, e^{-2\theta\sqrt{\lambda_1}y} \sum_{k=1}^{\infty} \hat{\varphi}_k^2 \qquad \forall y > 0 \, ,
\end{gathered}
\end{equation*}
which combined with \eqref{L2-aaa} yields \eqref{stima-L2}. Similarly, we have:
 \begin{equation}\label{L2-ccc}
 \begin{aligned}
\left\| \nabla\!_x E_R(\varphi) (\cdot,y) \right\|_{L^2(B_R)}^2 = \sum_{k=1}^{\infty} \lambda_k \, \hat{\varphi}_k^2 \, \psi_k^2(y) = & \, c_s^2 \sum_{k=1}^{\infty} \lambda_k \,  \hat{\varphi}_k^2 \left(\sqrt{\lambda_k} y \right)^{2s} K_s^2\!\left( \sqrt{\lambda_k} y \right) \\
\le & \, c_s^2 \sum_{k=1}^{\infty} \lambda_k y^2\,  \hat{\varphi}_k^2 \left(\sqrt{\lambda_k} y \right)^{2s} K_s^2\!\left( \sqrt{\lambda_k} y \right) \quad \forall y \ge 1 \, .
\end{aligned}
\end{equation}
Still by \eqref{eq304} we can infer that, up to relabeling $C_\theta$, there holds
\begin{equation}\label{est-theta-2}
c_s \, z^{1+s} K_s(z) \le C_\theta \, e^{-\theta z} \qquad \forall z \ge \sqrt{\lambda_1}  \, ,
\end{equation}
from which
\begin{equation}\label{L2-ddd}
\left\| \nabla\!_x E_R(\varphi) (\cdot,y) \right\|_{L^2(B_R)}^2 \le C_\theta^2 \, e^{-2\theta\sqrt{\lambda_1}y} \sum_{k=1}^{\infty} \hat{\varphi}_k^2 \qquad \forall y \ge 1 \, .
\end{equation}
As for the derivative w.r.t.~$y$, by virtue of \eqref{eq303} we obtain
 \begin{equation}\label{L2-eee}
\left\| \partial_y E_R(\varphi) (\cdot,y) \right\|_{L^2(B_R)}^2 = \sum_{k=1}^{\infty} \hat{\varphi}_k^2 \left[ \psi_k^\prime(y) \right]^2 =  c_s^2 \sum_{k=1}^{\infty} \lambda_k \,  \hat{\varphi}_k^2 \left(\sqrt{\lambda_k} y \right)^{2s} K_{1-s}^2\!\left( \sqrt{\lambda_k} y \right) \quad \forall y > 0 \, ,
\end{equation}
so that the estimate
\begin{equation}\label{L2-fff}
\left\| \partial_y E_R(\varphi) (\cdot,y) \right\|_{L^2(B_R)}^2 \le C_\theta^2 \, e^{-2\theta\sqrt{\lambda_1}y} \sum_{k=1}^{\infty} \hat{\varphi}_k^2 \qquad \forall y \ge 1
\end{equation}
follows by reasoning exactly as above, up to replacing $ K_s $ with $ K_{1-s} $ in \eqref{est-theta-2}. Hence \eqref{stima-Hs} is a consequence of \eqref{L2-ddd} and \eqref{L2-fff} (possibly relabelling $ C_\theta $). In order to prove \eqref{limit-0}, first of all we write
\begin{equation}\label{diff-l2}
\begin{aligned}
\left\| E_R(\varphi)(\cdot,y) - \varphi \right\|^2_{L^2(B_R)} = & \sum_{k=1}^{\infty} \left[ 1 - c_s \left(\sqrt{\lambda_k} y \right)^s K_s\!\left( \sqrt{\lambda_k} y \right) \right]^2 \hat{\varphi}_k^2 \\
= &  \sum_{k=1}^{N} \left[ 1 - c_s \left(\sqrt{\lambda_k} y \right)^s K_s\!\left( \sqrt{\lambda_k} y \right) \right]^2 \hat{\varphi}_k^2 \\
& +  \sum_{k=N+1}^{\infty} \left[ 1 - c_s \left(\sqrt{\lambda_k} y \right)^s K_s\!\left( \sqrt{\lambda_k} y \right) \right]^2 \hat{\varphi}_k^2 \\
\le &
\sum_{k=1}^{N} \left[ 1 - c_s \left(\sqrt{\lambda_k} y \right)^s K_s\!\left( \sqrt{\lambda_k} y \right) \right]^2 \hat{\varphi}_k^2 + (1+M)^2 \sum_{k=N+1}^{\infty} \hat{\varphi}_k^2
\end{aligned}
\end{equation}
for any $ N \in \mathbb{N} $, where we denote by $ M>0 $ the supremum of $ z \mapsto c_s \, z^s K_s(z) $. It is then plain that \eqref{limit-0} follows by letting first $ y \downarrow 0 $ (using \eqref{eq302}) and then $ N \to \infty $.

Identity \eqref{stima-integrale} is a standard one, see e.g.~\cite[Section 4]{Vaz12}. In any case, it could be proved here upon using \eqref{Hs}, integrating the identities in \eqref{L2-ccc} and \eqref{L2-eee} w.r.t.~$ y^{1-2s} \dy $ and suitably taking advantage of \eqref{eq302}--\eqref{eq303}.

Finally, let us establish \eqref{limit}. Formula \eqref{eq300} entails
\begin{equation}\label{der-y}
y^{1-2s}\frac{\partial E_R(\varphi)(x,y)}{\partial y}=\sum_{k=1}^{\infty} \hat\varphi_k \, \phi_k(x) \, y^{1-2s} \, \psi_k'(y) \qquad \forall (x,y) \in \mathcal{C}_R \, .
\end{equation}
On the other hand, thanks to \eqref{eq303}, for every $ k \in \mathbb{N} $ there holds
\begin{equation}\label{eq306}
y^{1-2s}\psi'_k(y)= -c_s \, \sqrt{\lambda_k} \, y^{1-2s} \left(\sqrt{\lambda_k}y\right)^s K_{1-s}\!\left(\sqrt{\lambda_k}y\right) .
\end{equation}
Let $\alpha_s:=\min\!\left\{1-s, \tfrac{1}{2} \right\}$. Property \eqref{eq304} yields
\begin{equation}\label{eq307}
 K_{1-s}(z) \leq C \, z^{-\alpha_s} \, e^{-z} \qquad \forall z\geq 1 \, , \qquad \text{with } C:=e\,K_{1-s}(1) \, .
\end{equation}
Hence from \eqref{eq306}--\eqref{eq307} we can infer that
\begin{equation}\label{eq308}
y^{1-2s}\left|\psi'_k(y)\right| \leq c_s \, C \max_{z\geq 1}\left\{z^{1-s-\alpha_s}e^{-z}\right\}\lambda_k^s =: C^\prime \lambda_k^s \qquad \forall y \ge \frac{1}{\sqrt{\lambda_k}} \, ;
\end{equation}
similarly, recalling \eqref{eq302}, we have:
\begin{equation}\label{eq309}
\begin{gathered}
y^{1-2s}\left|\psi'_k(y)\right| \le c_s \, \lambda^s_k \left(\sqrt{\lambda_k} y\right)^{1-s} K_{1-s}\!\left(\sqrt{\lambda_k} y\right)\leq c_s \, \lambda_k^s \sup_{z\in (0, 1]}z^{1-s}K_{1-s}(z) =: C'' \lambda_k^s \\
\forall y \le \frac{1}{\sqrt{\lambda_k}} \, .
\end{gathered}
\end{equation}
Let $ \hat{C} := C' \vee C'' $. Arguing as in \eqref{diff-l2}, by virtue of \eqref{der-y}--\eqref{eq306} and \eqref{eq308}--\eqref{eq309}, recalling also \eqref{eq305}, we obtain
\[
\begin{aligned}
& \left\| (-\Delta)^s_R \, \varphi + \mu_s \, y^{1-2s}\frac{\partial E_R(\varphi)(\cdot, y)}{\partial y} \right\|^2_{L^2(B_R)}
\\
\leq  & \sum_{k=1}^N \left[ \lambda_k^s - \mu_s \, c_s \, \sqrt{\lambda_k} \, y^{1-2s} \left(\sqrt{\lambda_k}y\right)^s K_{1-s}\!\left(\sqrt{\lambda_k}y\right) \right]^2 \hat\varphi_k^2 +  \left(1+\hat{C}\right)^2 \sum_{k={N+1}}^{\infty} \lambda_k^{2s} \, \hat\varphi_k^2
 \end{aligned}
\]
for any $N\in \mathbb N$, whence \eqref{limit} can be established upon letting first $ y \downarrow 0 $, using \eqref{eq302}, \eqref{ee11}, \eqref{eq0}, and then letting $N\to\infty$, using \eqref{H2s}.
\end{proof}

\subsection{Proof of Lemma \ref{lem6}}\label{App2}

To begin with, we introduce a family of cut-off functions that are helpful to many purposes. That is, let $\eta\in C^\infty\!\left(\overline{\Omega}\right)$ satisfy
\begin{equation}\label{cutoff}
0 \le \eta \le 1 \quad \text{in} \ \overline{\Omega} \, , \quad \eta = 1 \quad \text{in} \ \Omega_1 \, , \quad \eta = 0 \quad \text{in} \ \Omega_2^c \, , \quad y^{-2s} \partial_y \eta \in L^\infty(\Omega) \, ,
\end{equation}
where for each $R>0 $ the set $ \Omega_R $ is defined in \eqref{eq10}. For every $ k \in \mathbb{N} $, we then set
\begin{equation}\label{cutoff-k}
\eta_k(x,y):=\eta\!\left(\frac{x}{k},\frac{y}{k}\right) \qquad \forall(x,y) \in \overline{\Omega} \, .
\end{equation}

\begin{proof}[Proof of Lemma \ref{lem6}]
Let $ k \in \mathbb{N} $, with $ k>R $. A straightforward computation shows that
\[
\dive\!\left[ y^{1-2s} \nabla\!\left(E(v) \eta_k\right) \right] = 2 y^{1-2s} \langle \nabla E(v), \nabla \eta_k \rangle + E(v) \dive\!\left(y^{1-2s}\nabla \eta_k \right) \qquad \text{in } \Omega \, ,
\]
whence by testing the above identity against $  E(\varphi) - E_R(\varphi) $ and integrating by parts, with $ E_R(\varphi) $ extended to zero in $ \Omega\setminus \mathcal C_R $, we have:
\begin{equation}\label{eq251}
\begin{aligned}
& \, \int_{\Omega} \left\langle \nabla E(v), \nabla E_R(\varphi) -\nabla E(\varphi) \right\rangle \eta_k \, y^{1-2s}  \dx \dy \\
= & \, 2 \int_{\Omega} \langle \nabla E(v), \nabla \eta_k  \rangle \left( E(\varphi) - E_R(\varphi) \right) y^{1-2s} \dx \dy \\
 & + \int_{\Omega} E(v) \dive\!\left( y^{1-2s}\nabla \eta_k \right) \left( E(\varphi) - E_R(\varphi) \right) \dx \dy \\
 & + \int_{\Omega} E(v)  \left\langle \nabla \eta_k , \nabla E(\varphi) - \nabla E_R(\varphi) \right\rangle y^{1-2s} \dx \dy \, ,
\end{aligned}
\end{equation}
where we have used the fact that $ E(\varphi) $ and $ E_R(\varphi) $ have the same trace on $ \partial \Omega $, and implicitly a local approximation of $ E(v) $ by regular functions. Hence, from \eqref{eq251} we deduce
\begin{equation}\label{eq252}
\begin{aligned}
\int_{\mathcal{C}_R} \left\langle \nabla E(v), \nabla E_R(\varphi) \right\rangle \eta_k \, y^{1-2s}  \dx \dy = & \, 2\int_{\Omega} \langle \nabla E(v), \nabla \eta_k  \rangle \left( E(\varphi) - E_R(\varphi) \right) y^{1-2s} \dx \dy \\
 & + \int_{\Omega} E(v) \dive\!\left(y^{1-2s}\nabla \eta_k \right) \left( E(\varphi) - E_R(\varphi) \right) \dx \dy  \\
 & + \int_{\Omega} E(v)  \left\langle \nabla \eta_k , \nabla E(\varphi) \right\rangle y^{1-2s} \dx \dy  \\
 & - \int_{\mathcal{C}_R} E(v)  \left\langle \nabla \eta_k , \nabla E_R(\varphi)  \right\rangle y^{1-2s} \dx \dy \\
 & + \int_{\Omega} \left\langle \nabla E(v), \nabla E(\varphi) \right\rangle \eta_k \, y^{1-2s}  \dx \dy \, .
\end{aligned}
\end{equation}
Integrating by parts the first and the last term in the r.h.s.~of \eqref{eq252}, along with \eqref{u2}, yields
\begin{equation}\label{eq252-b}
\begin{aligned}
\int_{\mathcal{C}_R} \left\langle \nabla E(v), \nabla E_R(\varphi) \right\rangle \eta_k \, y^{1-2s}  \dx \dy = & -\int_{\Omega} E(v) \dive\!\left(y^{1-2s}\nabla \eta_k \right) E(\varphi) \, \dx \dy\\
& + \int_{\mathcal{C}_R} E(v) \dive\!\left(y^{1-2s}\nabla \eta_k \right) E_R(\varphi) \, \dx \dy \\
 & - 2 \int_{\Omega} E(v)  \left\langle \nabla \eta_k , \nabla E(\varphi)  \right\rangle y^{1-2s} \dx \dy  \\
 & + \int_{\mathcal{C}_R} E(v)  \left\langle \nabla \eta_k , \nabla E_R(\varphi)  \right\rangle y^{1-2s} \dx \dy \\
&+\frac 1{\mu_s}\int_{\mathbb R^d} v(x) \, (-\Delta)^s \varphi(x) \, \eta_k(x,0) \, \dx \, .
\end{aligned}
\end{equation}
Recalling \eqref{eq52}--\eqref{eq53}, \eqref{cutoff}--\eqref{cutoff-k} and the uniform boundedness of $ E(v) $, it is not difficult to obtain the following estimates:
\begin{equation*}
\begin{gathered}
\left| E(v)  \left\langle \nabla \eta_k , \nabla E(\varphi)  \right\rangle y^{1-2s} \right| \leq \frac{C}{k \left| (x,y) \right|^{d+2s}} \, \chi_{\Omega_{2k}\setminus \Omega_k} \le \frac{C}{\left| (x,y) \right|^{d+2s+1}} \, \chi_{\Omega_{2k}\setminus \Omega_k}   \, , \\
\left| E(v) \dive\!\left(y^{1-2s}\nabla \eta_k \right) E(\varphi) \right| \leq C \, \frac{k+y}{k^2 \left| (x,y) \right|^{d+2s}} \, \chi_{\Omega_{2k}\setminus \Omega_k} \le \frac{C}{\left| (x,y) \right|^{d+2s+1}} \, \chi_{\Omega_{2k}\setminus \Omega_k} \, ,
\end{gathered}
\end{equation*}
where $ C>0 $, here and below, is a suitable constant independent of $k$. Since
$$ (x,y) \mapsto \frac{1}{\left|(x,y)\right|^{d+2s+1}} \in L^1\!\left( \Omega_1^c \right) , $$
we can assert that
\begin{equation}\label{e100}
\lim_{k \to \infty} \left| \int_{\Omega} E(v) \dive\!\left(y^{1-2s}\nabla \eta_k \right) E(\varphi) \, \dx \dy \right| + \left| \int_{\Omega} E(v)  \left\langle \nabla \eta_k , \nabla E(\varphi)  \right\rangle y^{1-2s} \dx \dy \right|  =  0 \, .
\end{equation}
Thanks to the assumptions on $ \nabla E(v) $, the uniform boundedness of $ E(v) $ and
$$
E_R(\varphi) , \nabla E_R(\varphi) \in L^2\!\left(\mathcal C_R ; e^{\sqrt{\lambda_1}y} y^{1-2s} \dx\dy \right) ,
$$
which trivially follow from \eqref{stima-L2}, \eqref{stima-Hs}, \eqref{stima-integrale}, it is apparent that
\begin{equation}\label{L2-a}
\left\langle \nabla E(v), \nabla E_R(\varphi) \right\rangle , \ E(v)E_R(\varphi) , \ E(v)\nabla E_R(\varphi) \in L^1\!\left(\mathcal C_R ; y^{1-2s} \dx\dy \right) .
\end{equation}
Upon observing that
\begin{equation*}
\begin{gathered}
\left| E(v) \dive\!\left(y^{1-2s}\nabla \eta_k \right) E_R(\varphi) \right| \le \frac{C}{k^2} \left| E(v) E_R(\varphi) \right| \chi_{\left(\Omega_{2k} \setminus \Omega_k \right) \cap \mathcal{C}_R} \, y^{1-2s} \, , \\
\left| E(v)  \left\langle \nabla \eta_k , \nabla E_R(\varphi)  \right\rangle \right| y^{1-2s} \le \frac{C}{k} \left| E(v) \nabla E_R(\varphi) \right| \chi_{\left(\Omega_{2k} \setminus \Omega_k \right) \cap \mathcal{C}_R} \, y^{1-2s} \, ,
\end{gathered}
\end{equation*}
by exploiting \eqref{e100} and \eqref{L2-a} we can finally let $ k \to \infty $ in \eqref{eq252-b}, which yields \eqref{eq250}.
\end{proof}

\subsection{Proof of Lemma \ref{lem7}}\label{App3}

We first define another family of useful cut-off functions. Let $ \xi \in C^\infty([0,\infty)) $ satisfy
\begin{equation}\label{cutoff-y}
0 \le \xi \le 1 \quad \text{in} \ [0,\infty) \, , \qquad \xi = 1 \quad \text{in} \ [0,1] \, , \qquad \xi = 0 \quad \text{in} \ [2,\infty]\,.
\end{equation}
For every $R\ge 1$ put
\begin{equation}\label{eq:cutoff}
\gamma_R(x):=\xi\!\left(\frac{|x|}{R}\right) \qquad \forall x \in \mathbb{R}^d \,.
\end{equation}

\begin{proof}[Proof of Lemma \ref{lem7}]
It is convenient to first work with $ \{ v_\varepsilon \}_{\varepsilon>0} $, the latter being the convolution of $ v $ against a standard positive, compactly supported, regular kernel $\eta_\varepsilon$. Since we know that $ v , (-\Delta)^s v \in L^\infty(\mathbb{R}^d) $, one can show that
\begin{equation}\label{lap-conv}
\left[ (-\Delta)^s v \right]_\varepsilon = (-\Delta)^s (v_\varepsilon)
\end{equation}
and
\begin{equation}\label{lap-conv-2}
\lim_{R \to \infty} \int_{\mathbb{R}^d} (-\Delta)^s\! \left(v_\varepsilon \gamma_R \right) \varphi \, \dx = \int_{\mathbb{R}^d} (-\Delta)^s\! \left(v_\varepsilon \right) \varphi \, \dx \qquad \forall \varphi \in C^\infty_c(\mathbb{R}^d) \, .
\end{equation}
For a rigorous justification of these facts, we refer e.g.~to the proofs of \cite[Lemma 3.5 and Proposition 3.6]{Mur}. Given any $ R,r \ge 1 $, a simple integration by parts that exploits the regularity of the functions involved along with \eqref{eq-rho-2} yields
\begin{equation}\label{energy-1}
\begin{aligned}
\, & \int_\Omega \left| \nabla E\!\left( v_\varepsilon \gamma_R \right) \right|^2 \left( \gamma_r \rho_\alpha \right) y^{1-2s} \dx\dy \\
= \, & -\int_\Omega \left\langle \nabla E\!\left( v_\varepsilon \gamma_R \right) , \nabla \! \left( \gamma_r \rho_\alpha \right) \right\rangle E\!\left( v_\varepsilon \gamma_R \right) y^{1-2s}\dx\dy + \frac{1}{\mu_s} \int_{\mathbb{R}^d} (-\Delta)^s\! \left(v_\varepsilon \gamma_R \right) v_\varepsilon \gamma_R \gamma_r \, \dx \, .
\end{aligned}
\end{equation}
With no loss of generality we can and will assume that $ \gamma_r $ complies with
\begin{equation}\label{y1}
\left| \nabla \gamma_r(x) \right|^2 \le C \, \gamma_r(x) \qquad \forall x \in \mathbb{R}^d \, ,
\end{equation}
since $C$ is large enough by assumption. Thanks to Young's inequality, \eqref{eq-rho-1} and \eqref{y1}, we obtain:
\begin{equation}\label{young}
\begin{aligned}
& \left| \int_\Omega \left\langle \nabla E\!\left( v_\varepsilon \gamma_R \right) , \nabla \! \left( \gamma_r \rho_\alpha \right) \right\rangle E\!\left( v_\varepsilon \gamma_R \right) y^{1-2s} \dx\dy \right| \\
\le & \, \frac{1}{2} \int_\Omega \left| \nabla E\!\left( v_\varepsilon \gamma_R \right) \right|^2 \left( \gamma_r \rho_\alpha \right) y^{1-2s} \dx\dy + \left( \kappa+C^2 \right) \int_{\mathcal{C}_{2r}} \left| E\!\left( v_\varepsilon \gamma_R \right) \right|^2 \rho_\alpha \, y^{1-2s} \dx\dy \, .
\end{aligned}
\end{equation}
Hence \eqref{energy-1}--\eqref{young} imply
\begin{equation}\label{energy-2}
\begin{aligned}
\, & \int_{\mathcal{C}_r} \left| \nabla E\!\left( v_\varepsilon \gamma_R \right) \right|^2 \rho_\alpha \, y^{1-2s}  \dx\dy \\
\le \, & \frac{2}{\mu_s} \int_{\mathbb{R}^d} (-\Delta)^s\! \left(v_\varepsilon \gamma_R \right) v_\varepsilon \gamma_R \gamma_r \, \dx + 2 \left( \kappa + C^2 \right) \| v \|_\infty^2 \left|B_{2r}\right| \int_0^{\infty} \rho_\alpha(y) \, y^{1-2s} \, \dy  \, .
\end{aligned}
\end{equation}
Finally, estimate \eqref{est-energy} follows by letting first $ R \to \infty $ and then $ \varepsilon \downarrow 0 $ in \eqref{energy-2}, upon exploiting \eqref{lap-conv-2}, \eqref{lap-conv}, the local convergence of the convolution (e.g.~in $ L^2_{\mathrm{loc}}(\mathbb{R}^d) $) and the plain fact that the extension operator is stable under all of these passages.
\end{proof}

\par\bigskip\noindent
\textbf{Acknowledgments.}   GG is partially supported by the PRIN project ``Equazioni alle derivate parziali di tipo ellittico e parabolico: aspetti geometrici, disuguaglianze collegate, e applicazioni'' (Italy). All authors are members of the ``Gruppo Nazionale per l'Analisi Matematica, la Probabilit\`a e le loro Applicazioni'' (GNAMPA) of the ``Istituto Nazionale di Alta Matematica'' (INdAM, Italy). We are grateful to Yannick Sire for helpful discussions.


\begin{thebibliography}{999}

\scriptsize

\bibitem{Ab} M. Abramowitz, I.A. Stegun, ``Handbook of Mathematical Functions with Formulas, Graphs,
and Mathematical Tables'', National Bureau of Standards Applied Mathematics Series, 55. For sale by the Superintendent of Documents, U.S. Government Printing Office, Washington, D.C., 1964.

\bibitem{AMS} L. Ambrosio, A. Mondino, G. Savar\'e, \emph{Nonlinear diffusion equations and curvature conditions in metric measure spaces}, Mem. Amer. Math Soc. (in press), \url{https://doi.org/10.1090/memo/1270}.

\bibitem{AiB} B. Andreianov, M. Brassart, \emph{Uniqueness of entropy solutions to fractional conservation laws with ``fully infinite'' speed of propagation}, J. Differential Equations (in press), \url{https://doi.org/10.1016/j.jde.2019.10.008}.

\bibitem{ACP} D. Aronson, M.G. Crandall, L.A. Peletier, \emph{Stabilization of solutions of a degenerate nonlinear diffusion problem}, Nonlinear Anal. \bf 6 \rm (1982), 1001--1022.

\bibitem{AC} I. Athanasopoulos, L.A. Caffarelli, \emph{Continuity of the temperature in boundary heat control problems}, Adv. Math. \bf 224 \rm (2010), 293--315.

\bibitem{BPSV} B. Barrios, I. Peral, F. Soria, E. Valdinoci, \emph{A Widder's type theorem for the heat equation with nonlocal diffusion},
Arch. Ration. Mech. Anal. \textbf{213} (2014),  629--650.

\bibitem{BCP} P. B\'enilan, M.G. Crandall, M. Pierre, \emph{Solutions of the porous medium equation in $\mathbb{R}^N$ under optimal conditions on initial values}, Indiana Univ. Math. J. \bf 33 \rm (1984), 51--87.

\bibitem{BGT} M. Bologna, P. Grigolini, C. Tsallis, \emph{Anomalous diffusion associated with nonlinear fractional derivative Fokker-Planck-like equation: exact time-dependent solutions}, Physical Review E \textbf{62} (2000), 2213.


\bibitem{BFV} M. Bonforte, A. Figalli, J.L. V\'azquez, \emph{Sharp global estimates for local and nonlocal porous medium-type equations in bounded domains}, Anal. PDE \textbf{11} (2018), 945--982.

\bibitem{BSV} M. Bonforte, Y. Sire, J.L. V\'azquez, \emph{Existence, uniqueness and asymptotic behaviour for fractional porous medium equations on bounded domains}, Discrete Contin. Dyn. Syst. \textbf{35} (2015),  5725--5767.

\bibitem{BSV2}  M. Bonforte, Y. Sire, J.L. V\'azquez, \emph{Optimal existence and uniqueness theory for the fractional heat equation}, Nonlinear Anal. \textbf{153} (2017), 142--168.

\bibitem{BVadv} M. Bonforte, J.L. V\'azquez, \emph{Quantitative local and global a priori estimates for fractional nonlinear
diffusion equations}, Adv. Math. \textbf{250} (2014), 242--284.

\bibitem{BV1} M. Bonforte, J.L. V\'azquez,  \emph{A priori estimates for fractional nonlinear degenerate diffusion equations on bounded domains}, Arch. Ration. Mech. Anal. \textbf{218} (2015), 317--362.

\bibitem{BV2} M. Bonforte, J.L. V\'azquez, \emph{Fractional nonlinear degenerate diffusion equations on bounded domains part I. Existence, uniqueness and upper bounds}, Nonlinear Anal. \textbf{131} (2016), 363--398.

\bibitem{BC} H. Br\'ezis, M.G. Crandall, \emph{Uniqueness of solutions of the initial-value problem for $u_t-\Delta\varphi(u)=0$}, J. Math. Pures Appl. \bf 58 \rm (1979), 153--163.

\bibitem{CSi} X. Cabr\'e, Y. Sire, \emph{Nonlinear equations for fractional Laplacians, I: Regularity, maximum principles, and Hamiltonian estimates}, Ann. Inst. H. Poincar\'e Anal. Non Lin\'eaire \textbf{31} (2014), 23--53.

\bibitem{CS} L.A. Caffarelli, L. Silvestre, \emph{An extension problem related to the fractional Laplacian}, Comm. Partial Differential Equations \textbf{32} (2007), 1245--1260.

\bibitem{CL} M.G. Crandall, T.M. Liggett, \emph{Generation of semi-groups of nonlinear transformations on general Banach spaces}, Amer. J. Math. \textbf{93} (1971), 265--298.

\bibitem{BD} E.B. Davies, ``Heat Kernels and Spectral Theory'', Cambridge Tracts in Mathematics, 92. Cambridge University Press, Cambridge, 1990.

\bibitem{Vaz11} A. de Pablo, F. Quir\'os, A. Rodr\'iguez, J.L. V\'azquez, \emph{A fractional porous medium equation}, Adv. Math. \bf 226 \rm (2011), 1378--1409.

\bibitem{Vaz12} A. de Pablo,  F. Quir\'os,  A. Rodr\'iguez,  J.L. V\'azquez, \emph{A general fractional porous medium equation}, Comm. Pure Appl. Math. \bf 65 \rm (2012), 1242--1284.

\bibitem{TJJ} F. del Teso, J. Endal, E.R. Jakobsen, \emph{Uniqueness and properties of distributional solutions of nonlocal equations of porous medium type}, Adv. Math. \textbf{305} (2017), 78--143.

 \bibitem{TJJ3} F. del Teso, J. Endal, E.R. Jakobsen, \emph{On the well-posedness of solutions with finite energy for nonlocal equations of porous medium type}, in: ``Non-Linear Partial Differential Equations, Mathematical Physics, and Stochastic Analysis'', pp. 129--168, EMS Series of Congress Reports, 2018.

\bibitem{TJJ1} F. del Teso, J. Endal, E.R. Jakobsen, \emph{Robust numerical methods for nonlocal (and local) equations of porous medium type. Part I: Theory}, SIAM J. Numer. Anal. 57 (2019), 2266--2299.

\bibitem{TJJ2} F. del Teso, J. Endal, E.R. Jakobsen, \emph{Robust numerical methods for nonlocal (and local) equations of porous medium type. Part II: Schemes and experiments}, SIAM J. Numer. Anal. \textbf{56} (2018), 3611--3647.

\bibitem{EK} D. Eidus, S. Kamin, \emph{The filtration equation in a class of functions decreasing at infinity},
Proc. Amer. Math. Soc. \textbf{120} (1994), 825--830.

\bibitem{GMP} G. Grillo, M. Muratori, F. Punzo, \emph{Fractional porous media equations: existence and uniqueness of weak solutions with measure data}, Calc. Var. Partial Differential Equations \textbf{54} (2015), 3303--3335.

\bibitem{GMPjmpa} G. Grillo, M. Muratori, F. Punzo, \emph{The porous medium equation with large initial data on negatively curved Riemannian manifolds}, J. Math. Pures Appl. \textbf{113} (2018), 195--226.

\bibitem{GMPjems} G. Grillo, M. Muratori, F. Punzo, \emph{The porous medium equation with measure data on negatively curved Riemannian manifolds}, J. Eur. Math. Soc. (JEMS) \textbf{20} (2018), 2769--2812.


\bibitem{JKO} M. Jara, T. Komorowski, S. Olla, \emph{Limit theorems for additive functionals of a Markov chain}, Ann. Appl. Probab. \textbf{19} (2009), 2270--2300.

\bibitem{JLS} M. Jara, C. Landim, S. Sethuraman, \emph{Nonequilibrium fluctuations for a tagged particle in mean-zero one-dimensional zero-range processes}, Probab. Theory Relat. Fields \textbf{145} (2009), 565--590.

\bibitem{KKT} S. Kamin, R. Kersner, A. Tesei, \emph{On the Cauchy problem for a class of parabolic equations with variable density}, Atti Accad. Naz. Lincei Cl. Sci. Fis. Mat. Natur. Rend. Lincei (9) Mat. Appl. \textbf{9} (1998), 279--298.



\bibitem{LMT} E.K. Lenzi, R.S. Mendes, C. Tsallis, \emph{Crossover in diffusion equation: Anomalous and normal behaviors}, Physical Review E \textbf{67} (2003), 031104.

\bibitem{Mur} M. Muratori, \emph{The fractional Laplacian in power-weighted $L^p$ spaces: integration-by-parts formulas and self-adjointness}, J. Funct. Anal. \textbf{271} (2016), 3662--3694.

\bibitem{Noch} R.H. Nochetto, E. Ot\'arola, A.J. Salgado, \emph{A PDE approach to fractional diffusion in general domains: a priori error analysis}, Found. Comput. Math. \textbf{15} (2015), 733--791.

\bibitem{P} A. Pazy, ``Semigroups of Linear Operators and Applications to Partial Differential Equations'', Applied Mathematical Sciences, 44. Springer-Verlag, New York, 1983.

\bibitem{Pierre} M. Pierre, \emph{Uniqueness of the solutions of $u_t-\Delta\varphi(u)=0$ with initial datum a measure}, Nonlinear Anal. {\bf 6} (1982), 175--187.

\bibitem{Pu} F. Punzo, \emph{On the Cauchy problem for nonlinear parabolic equations with variable density}, J. Evol. Equ. \textbf{9} (2009), 429--447.

\bibitem{PV}  F. Punzo, E. Valdinoci, \emph{Uniqueness in weighted Lebesgue spaces for a class of fractional parabolic and elliptic equations}, J. Differential Equations {\bf 258} (2015), 555--587.

\bibitem{ST} P.R. Stinga, J.L. Torrea, \emph{Extension problem and Harnack's inequality for some fractional operators}, Comm. Partial Differential Equations \textbf{35} (2010), 2092--2122.

\bibitem{Vbook} J.L. V\'azquez, ``The Porous Medium Equation. Mathematical Theory'', Oxford Mathematical Monographs. The Clarendon Press, Oxford University Press, Oxford, 2007.

\bibitem{VR} J.L. V\'azquez, \emph{Recent progress in the theory of nonlinear diffusion with fractional Laplacian operators}, Discrete Contin. Dyn. Syst. Ser. S \textbf{7} (2014), 857--885.

\bibitem{VR2} J.L. V\'azquez, \emph{The mathematical theories of diffusions: nonlinear and fractional diffusion}, in: ``Nonlocal and nonlinear diffusions and interactions: new methods and directions''. M. Bonforte,  G. Grillo, editors. Lecture Notes in Mathematics, \textbf{2186}. Fondazione CIME/CIME Foundation Subseries. Springer, Cham; Fondazione C.I.M.E., Florence, 2017.

\end{thebibliography}
\end{document}